\documentclass[reqno, 10pt]{amsart}
\usepackage[margin=1in]{geometry}

\usepackage{amsmath,amssymb,amscd,graphicx, color, tensor,hyperref,tikz,stackengine}
\usepackage{setspace}
\usepackage{cancel}

\newcommand\oast{\stackMath\mathbin{\stackinset{c}{0ex}{c}{0ex}{\ast}{\bigcirc}}}
\newcommand{\R}{{ \mathbb{R}  }}
\newcommand{\N}{{ \mathbb{N}  }} 
\newcommand{\Z}{ {\mathbb{Z}}}

\newcommand{\calP}{{ \mathcal P  }}

\newcommand{\calW}{{\mathcal W}}

\newcommand{\pa}{\partial}

\newcommand{\Om}{{ \Omega  }}

\newcommand{\na}{\nabla}

\newcommand {\ga}{\gamma}
\newcommand {\al}{\alpha}
\newcommand {\be}{\beta}

\newcommand{\del}{\delta}

\newcommand {\brho}{\bar \rho}
\newcommand{\brt}{\brho_t}
\newcommand{\xt }{X_t}
\newcommand{\yt }{Y_t}
\newcommand{\xit }{X^i_t}
\newcommand{\xis }{X^i_s}

\newcommand{\xjt }{X^j_t}
\newcommand{\xjs }{X^j_s}

\newcommand{\bx}{\mathbf{x}}

\newcommand{\by}{\mathbf{y}}

\newcommand{\EE}{\mathbb{E}}

\newtheorem{maintheorem}{Theorem}

\newcommand{\T}{\mathbb{T}}
\newcommand{\lt}{\left}
\newcommand{\rt}{\right}
\newcommand{\bq}{\begin{equation}}
\newcommand{\eq}{\end{equation}}
 \newcommand{\dx}{\textnormal{d}x}
 \newcommand{\Law}{\textnormal{Law}}
\def\d{\mathrm{d}}

\begin{document}

\newtheorem{definition}{Definition}[section]
\newtheorem{lemma}{Lemma}[section]
\newtheorem{proposition}{Proposition}[section]
\newtheorem{theorem}{Theorem}[section]
\newtheorem{assumption}{Assumption}[section]
\newtheorem{cor}{Corollary}[section]
\newtheorem{remark}{Remark}[section]
\numberwithin{equation}{section}
\newenvironment{pfthm1}{{\par\noindent\bf
           Proof of Theorem \ref{GETHM}. }}{\hfill\fbox{}\par\vspace{.2cm}}
\newenvironment{pfthm2}{{\par\noindent\bf
           Proof of Theorem \ref{GETHM1}. }}{\hfill\fbox{}\par\vspace{.2cm}}

\title[Propagation of chaos and bifurcation in two-type adhesion]{Uniform-in-time propagation of chaos and bifurcation in two-type adhesion systems}

\keywords{Stochastic interacting particle systems, diffusion--jump processes, propagation of chaos, long-time behavior,  bifurcation of equilibria.}

\begin{abstract}
We study a nonlocal adhesion model for two interacting tumor cell phenotypes, combining diffusion, pairwise interactions, and random phenotypic switching. The system admits a microscopic diffusion--jump particle description whose mean-field limit is a nonlinear McKean--Vlasov equation on a product space encoding position and internal state. We first establish uniform-in-time propagation of chaos in the weak-interaction regime using a coupling approach that combines reflection coupling for the diffusion with an optimal coupling of the spin-flip dynamics. As a byproduct, we obtain exponential long-time contraction for the nonlinear McKean--Vlasov equation in the first-order Wasserstein distance, implying uniqueness of the stationary distribution. We also investigate the complementary regime of strong interactions, where the homogeneous equilibrium may lose stability through a bifurcation mechanism.
  \end{abstract}

\author{Myeongju Chae}%
\address{School of applied mathematics and computer engineering, Hankyong University, Anseong, 17579, Republic of Korea}%
\email{mchae@hknu.ac.kr}
\author{Young-Pil Choi}
\address{Department of Mathematics, Yonsei University, Seoul, 03722, Republic of Korea}%
\email{ypchoi@yonsei.ac.kr}
  \maketitle
  
  \tableofcontents
%
%
%
%
%
%

\section{Introduction}
  
Models for tumor invasion and cell migration often combine diffusive motion, nonlocal interactions, and reaction mechanisms accounting for phenotypic changes. Such frameworks are designed to capture the collective behavior of heterogeneous cell populations, where spatial organization and internal state transitions evolve on comparable time scales \cite{ACN,BLM,CA,CL,PB}.
 
In the present work, we study a mean-field description of a nonlocal adhesion model involving two interacting cancer cell phenotypes. We consider two populations evolving on the $d$-dimensional flat torus $\T^d=\R^d/\Z^d$, where cells may switch between phenotypic states while interacting through adhesion forces. Let $u=u(x,t)$ and $v=v(x,t)$ denote the spatial densities of the two phenotypes at position $x\in\T^d$ and time $t>0$. The dynamics are governed by the following coupled system:
\begin{align}\label{pde}
\begin{aligned}
\pa_t u & =  \frac{\sigma^2(1)}{2} \Delta u - \na \cdot ( u ( \nabla U * u + \nabla V * v))-\al_1 u +\al_{-1} v, \\
\pa_t v & =  \frac{\sigma^2(-1)}{2} \Delta v - \na \cdot ( v ( \nabla  U * u + \nabla V * v))+ \al_1 u-\al_{-1} v,
\end{aligned}
\end{align}
where $\sigma(1),\sigma(-1)>0$ denote the diffusion coefficients associated with each phenotype, $U$ and $V$ are prescribed interaction potentials, and $\al_1,\al_{-1}\ge0$ are transition rates, not both zero. Throughout the paper, we assume 
\[
\int_{\T^d} u(x,t)\d x + \int_{\T^d} v(x,t) \d x=1, \quad \forall \, t \ge 0,
\]
since the system \eqref{pde} preserves the total mass. 

System \eqref{pde} provides a diffusive nonlocal interaction model for two tumor cell phenotypes, in which spatial motion, adhesion effects, and phenotypic transitions are coupled. Related nonlocal adhesion models posed on bounded domains have been introduced and their well-posedness properties analyzed in, for instance, \cite{ACL,BCE17}. The convolution terms describe cell-to-cell adhesion forces, which are assumed to depend linearly on the surrounding population densities, following standard modeling assumptions in nonlocal adhesion theory \cite{APS06}. We assume that the interaction potentials $U$ and $V$ generate odd interaction forces, namely
\[
\nabla U(-x) = -\nabla U(x), \quad \nabla V(-x) = -\nabla V(x), \quad \forall\,x \in \T^d.
\]
This symmetry reflects the action--reaction principle at the level of pairwise interactions. In particular, the total interaction force is balanced, a property that is natural both from a physical viewpoint and in the context of cell--cell adhesion.

The linear exchange terms model random transitions between the two phenotypic states. In particular, the terms involving $\al_1$ correspond to transitions from the $u$-type to the $v$-type population with constant rate $\al_1>0$, a mechanism commonly associated with epithelial--mesenchymal transition processes in tumor dynamics \cite{Y17}. Conversely, the terms involving $\al_{-1}$ account for the reverse transition with rate $\al_{-1}>0$, which may represent mesenchymal--epithelial transition. Bidirectional phenotypic switching mechanisms of this form also appear in tumor invasion
models, see for instance \cite{KLSW17, KSSSL20}. Due to the mass exchange between the two populations, the individual masses of $u$ and $v$ are not conserved, while the total mass associated with \eqref{pde} remains invariant in time.

A stochastic particle interpretation of \eqref{pde} can be formulated in terms of $N$ interacting agents whose spatial positions evolve diffusively and whose phenotypic states switch randomly in time. This viewpoint provides a natural microscopic description of the coupled diffusion--reaction dynamics encoded in \eqref{pde}. Such a stochastic model was proposed in \cite{ACCL}, where the authors derived \eqref{pde} as the mean-field limit of the corresponding particle system as the number of particles $N$ tends to infinity.

In this framework, each particle is described by a pair $(X^i_t,Y^i_t)$, where $X^i_t\in\T^d$ denotes the spatial position of the $i$-th particle and
$Y^i_t\in\{1,-1\}$ its phenotypic state, for $i=1,\dots,N$ and $t\ge0$. The spatial dynamics are governed by
\bq\label{eq0}
  X^i_t = X^i_0 +\int_0^t \sigma(Y^i_s)\,\d B^i_s + \frac 1N \sum_{j=1}^N \int_0^t F(\xis-\xjs, Y^j_s)\,\d s, 
\eq
where $\{B^i_t\}_{i=1}^N$ are independent standard $d$-dimensional Brownian motions. The interaction kernel $F:\T^d\times\{\pm1\}\to\R^d$ encodes the adhesion force generated by a
neighboring particle and depends on the neighbor's phenotypic state. More precisely, we assume
\[
F(\cdot,1)=\nabla U, \quad F(\cdot,-1)=\nabla V,
\]
so that each particle experiences the superposition of forces generated by surrounding particles of the two phenotypes. In particular, the interaction force acting on a particle does not depend on its own phenotypic state, but on the spatial configuration and the types of its neighbors. This structure is consistent with the macroscopic drift term $\nabla U*u+\nabla V*v$ appearing in \eqref{pde}. Moreover, the oddness of $\nabla U$ and $\nabla V$
implies that $F(0,y)=0$ and enforces an action--reaction symmetry at the microscopic level.

The internal state variable $Y_t^i$ evolves according to an inhomogeneous Poisson-driven jump process,
\bq\label{eq1}
Y_t^i = Y_0^i + \int_0^t ( -1 - Y_{s-}^i )\, \d\tilde E^i_{1}(s) + \int_0^t ( 1 - Y_{s-}^i )\, \d\tilde E^i_{-1}(s),
\eq
where the jump intensities of the counting processes $\tilde E^i_{\pm1}$ are chosen so as
to reproduce the transition rates $\alpha_{\pm 1}$ appearing in \eqref{pde}:
\[
  \tilde E_1^i(t) = E^i \left( \al_1 \int_0^t \chi_1( Y^i_{s-}) \d s \right) \quad \mbox{and} \quad
    \tilde E_{-1}^i(t) = E^i \left( \al_{-1} \int_0^t \chi_{-1}( Y^i_{s-}) \d s \right).
    \]
Here $\{E^i(t)\}_{i=1}^N$ are independent unit-rate Poisson processes, and $\chi_k:\{1,-1\}\to\{0,1\}$, $k=\pm1$, is defined by $\chi_k(l)=1$ if $l=k$ and $\chi_k(l)=0$ otherwise. The stochastic integrals in \eqref{eq1} are understood in the standard sense of integration of bounded predictable processes against c\`adl\`ag semimartingales.

Mean-field limits for stochastic systems of interacting particles have been extensively investigated since the seminal work of \cite{Szn91}, which established propagation of chaos for the Vlasov--McKean equation under globally Lipschitz interaction assumptions. Extensions to non-Lipschitz interaction kernels have been studied in \cite{BCC11,CCHS19, CS,CS19,HS19,JW16}, while particle systems with singular interaction forces, both of first and second order, are analyzed in \cite{BJWpre,CCS19,HLP20,JW18,Sa19}. A comprehensive overview of mean-field limits for stochastic interacting particle systems can be found in \cite{CD22, JW17} and the references therein.

Most existing quantitative results in this literature focus on systems whose macroscopic description consists of a single equation, and in which propagation of chaos is typically established only on finite time intervals. By contrast, the system \eqref{pde} involves two interacting populations whose individual masses are not conserved due to the presence of linear exchange terms, although the total population remains invariant. This structural feature suggests that a faithful stochastic description of \eqref{pde} must simultaneously account for spatial motion and random transitions between internal states, and that new analytical tools are required to control these effects over long time horizons.

Stochastic models combining spatial dynamics with type-switching mechanisms have appeared in only a limited number of works. Models without diffusion were considered in \cite{Albi}, while diffusive systems with random type changes were studied in \cite{LLN20}. The Poisson-driven formulation of the internal state dynamics used in \eqref{eq1} is inspired by the stochastic framework introduced in \cite{LLN20}, where the authors study bimolecular reaction--diffusion systems in which particles change chemical type upon random encounters. Earlier contributions include stochastic interacting particle models driven by Poisson processes \cite{Oe}, as well as convergence results for jump Markov processes in the context of ordinary differential equations \cite{Ku}.

%
%
%
%
%
%

\subsection{Main results}

The goal of this paper is to develop a quantitative mean-field theory for a two-type interacting diffusion--jump particle system on the flat torus, and to connect microscopic coupling mechanisms with macroscopic stability properties in Wasserstein distance.  In particular, we establish (i) uniform-in-time propagation of chaos for the $N$-particle dynamics, (ii) exponential long-time contraction for the nonlinear McKean--Vlasov limit in the weak-interaction regime, and (iii) a complementary bifurcation criterion for stationary solutions when the interaction strength is large enough.

To formulate the main results in a precise and quantitative way, we first introduce the main assumptions and the distance framework used throughout the paper. Our analysis is naturally carried out on the product space
\[
\Pi := \T^d \times \{\pm1\},
\]
which encodes both the spatial position and the phenotypic state of particles. We slightly abuse notation and write $|x-w|$ for the distance on the flat torus $\T^d$, defined by
\[
|x-w| := \min_{n\in\Z^d} |x-w+n|.
\]
With this convention, the diameter of $\T^d$ is $\sqrt d/2$. In order to measure spatial discrepancies in a way compatible with the geometry of the torus, we introduce the auxiliary function $f:\R_+ \to [-1,1]$ defined by
\[
f(r)=\sin \left(\frac{\pi}{\sqrt d}\,r\right).
\]
Since $0\le |x|\le \sqrt d/2$ for $x\in\T^d$, the function $f(|x|)$ is strictly increasing in $|x|$ and is equivalent to the distance $|x|$ in the sense that
\[
\frac{2}{\sqrt d}|x|\le f(|x|)\le \frac{\pi}{\sqrt d}|x|.
\]

On the discrete set $\{\pm1\}$, we observe that
\[
|y-\bar y| = 2 \mathbf 1_{\{y\neq \bar y\}}, \quad y,\bar y\in\{\pm1\}.
\]
Thus, the absolute difference naturally induces a metric on the type space which takes the values $\{0,2\}$. 

These component-wise distances are combined to define a cost on the product space $\Pi$ by
\[
\d\big((x,y),(p,q)\big) := |x-p| + |y-q|
= |x-p| + 2\mathbf 1_{\{y\neq q\}}.
\]

We denote by $\calW_1$ the associated $1$-Wasserstein distance on $\calP(\Pi)$.  When restricted to the spatial marginals, $\calW_1$ induces the standard first-order Wasserstein distance on $\T^d$, namely
\[
\calW_1(\nu,\mu)=\inf_{\gamma\in\Pi(\nu,\mu)}\iint_{\T^d\times\T^d}|x-w|\,\d\gamma(x,w),
\]
where $\Pi(\nu,\mu)$ denotes the set of couplings of $\nu$ and $\mu$, that is, the set of probability measures on the product space whose marginals are $\nu$ and $\mu$.

The interaction force $F:\T^d\times\{\pm1\}\to\R^d$ is assumed to satisfy the following Lipschitz-type condition: there exists $\eta>0$ such that for all $x,w\in\T^d$ and $y\in\{\pm1\}$,
\bq\label{hyp_F}
|F(x,y)-F(w,y)| \le \eta  f(|x-w|).
\eq
In particular, since $F(\cdot,y)$ is odd on $\T^d$ for each $y\in\{\pm1\}$, we have $F(0,y)=0$ and hence,
by \eqref{hyp_F},
\[
|F(x,y)| = |F(x,y)-F(0,y)| \le \eta  f(|x|)\le \eta, \quad (x,y)\in \T^d\times\{\pm1\}.
\]

The propagation of chaos and contraction results are formulated in terms of a nonlinear mean-field limit on $\Pi$, which we now briefly describe. Due to conservation of the total mass, we consider probability measures on $\Pi$ and rewrite the macroscopic system \eqref{pde} in the unified form
\bq\label{intro_new_pde}
\pa_t\rho = \frac{\sigma^2(y)}{2}\Delta_x \rho - \na_x\cdot(\rho\, F\oast\rho) + {\rm T}\rho, \quad z=(x,y)\in\Pi.
\eq
Here the nonlocal drift is given by
\begin{align}\label{oast}
\begin{aligned}
(F\oast  \rho) (w)&:=\iint_\Pi F(w-x, y) \rho(x,y)\, \d x \d y\cr
&=   \int_{\T^d} F(w-x, 1) \rho (x, 1)\, \d x + \int_{\T^d} F(w-x, -1)  \rho (x, -1)\, \d x\cr
&=   \int_{\T^d} \nabla U(w-x) \rho (x, 1)\, \d x + \int_{\T^d} \nabla V(w-x)  \rho (x, -1)\, \d x
  \end{aligned} \end{align}
and ${\rm T}$ denotes the linear operator encoding random transitions between the two phenotypic states with rates $\alpha_1$ and $\alpha_{-1}$,
\[
{\rm T}\rho(x,y) = (-1)y \alpha_1 \rho(x,1)  + y  \alpha_{-1} \rho(x,-1).
\]
We denote by $\d z=\d x\otimes\d{}_\#$ the product of Lebesgue measure on $\T^d$ and the counting measure on $\{\pm1\}$.

Writing $\rho(x,1,t)=u(x,t)$ and $\rho(x,-1,t)=v(x,t)$, the equation above is equivalent to the original system \eqref{pde}. Equivalently, it can be interpreted as the evolution equation for the law of a nonlinear McKean--Vlasov process on $\Pi$, which provides the natural mean-field limit of the interacting particle system.

We now state the main quantitative conclusions of this work in an informal form, highlighting the structure of the estimates while suppressing non-essential constants and technical conditions. Precise statements, including sharp decay rates and explicit parameter restrictions, are given in Theorems \ref{thm_propa}, \ref{thm_lt}, and \ref{bifurcation} below.

Our first result concerns a quantitative and uniform-in-time mean-field theory for the two-type diffusion--jump particle system associated with \eqref{pde}. For the same microscopic model, propagation of chaos was obtained in \cite{ACCL} by means of a relative entropy method, but only on arbitrary finite time intervals. In contrast, we adopt a coupling-based approach that combines reflection coupling for the diffusion component with an optimal coupling of the type-switching dynamics. More precisely, we use two equivalent microscopic representations of the internal state process: a Poisson-driven jump SDE and the corresponding spin-flip Markov chain. A spin-flip formulation of the internal state dynamics was also used in \cite{Albi}. Since both representations are governed by the same flip rates, they share the same generator and therefore define the same c\`{a}dl\`{a}g Markov jump process in distribution. This flexibility allows us to tailor the coupling strategy to the structure of each component of the dynamics. The use of reflection coupling for the diffusion part aligns naturally with the Wasserstein structure introduced above, which simultaneously captures spatial separation and type disagreement. Reflection coupling was introduced in \cite{Eb} to establish exponential convergence for linear drift--diffusion equations and was later employed in \cite{DAGR} to obtain uniform-in-time propagation of chaos for the Vlasov--McKean equation. Inspired by these ideas, we extend the reflection coupling methodology to interacting particle systems with random type changes by coupling the spin-flip dynamics optimally at the microscopic level. This combined coupling strategy allows us to control the particle system uniformly in time and leads to the propagation of chaos estimates in Wasserstein distance that are valid over infinite time horizons. To the best of our knowledge, this is the first uniform-in-time propagation of chaos result for interacting particle systems involving both diffusion and random-type switching.
 
  For a probability measure $\mu\in\calP(\Pi)$, we denote by $\pi_x{}_\#\mu$ and $\pi_y{}_\#\mu$ its spatial and type marginals. 
 
\begin{maintheorem}[Uniform-in-time propagation of chaos]
Let $\mu^N$ be the empirical measure associated to the interacting diffusion--jump system \eqref{eq0}--\eqref{eq1}, that is,
\[
\mu_t^N := \frac1N \sum_{i=1}^N \delta_{(X_t^i,Y_t^i)} \in  \calP(\Pi),
\]
where $(X_t^i,Y_t^i)$ is a solution of \eqref{eq0}--\eqref{eq1}. Let $\bar\rho$ be a solution of \eqref{intro_new_pde}.  Assume that $F$ satisfies \eqref{hyp_F} and the interaction strength $\eta$ is sufficiently small. Then there exist constants $c>0$ and $C>0$, independent of $t$ and $N$, such that for all $t\ge0$,
\begin{align}\label{eq:intro_poc_x}
\begin{aligned}
\EE \calW_1(\pi_x{}_\#\mu_t^N,\pi_x{}_\#\bar\rho_t ) &\le C e^{-c t} \EE \calW_1 (\pi_x{}_\#\mu_0^N,\pi_x{}_\#\bar\rho_0)
+\frac{C}{\sqrt{N}} (1 + {\bf 1}_{\{d=2\}} \log N )   \cr
&\quad +\mathcal R_N(t)\EE \calW_1 (\pi_y{}_\#\mu_0^N,\pi_y{}_\#\bar\rho_0 ) 
\end{aligned}
\end{align}
and
\[
\EE\calW_1 (\pi_y{}_\#\mu_t^N,\pi_y{}_\#\bar\rho_t ) \le e^{-(\alpha_1+\alpha_{-1}) t}\EE\calW_1(\pi_y{}_\#\mu_0^N,\pi_y{}_\#\bar\rho_0).
\]
Here  $\mathcal R_N(t)$ is a transient term which decays exponentially fast in $t$. In particular, if $\pi_y{}_\#\mu_0^N=\pi_y{}_\#\bar\rho_0$, then \eqref{eq:intro_poc_x} reduces to the uniform-in-time rate
\[
\EE \calW_1(\pi_x{}_\#\mu_t^N,\pi_x{}_\#\bar\rho_t ) \le C e^{-c t} \EE \calW_1 (\pi_x{}_\#\mu_0^N,\pi_x{}_\#\bar\rho_0)
+\frac{C}{\sqrt{N}} (1 + {\bf 1}_{\{d=2\}} \log N ).
\]
\end{maintheorem}

As a byproduct of the uniform-in-time propagation of chaos established above, we obtain exponential long-time contraction for the associated nonlinear McKean--Vlasov equation in the weak-interaction regime. This result makes precise how macroscopic stability of the mean-field dynamics emerges from the underlying microscopic coupling structure, and in particular yields uniqueness of the stationary law and convergence to equilibrium in the $1$-Wasserstein distance.

\begin{maintheorem}[Long-time exponential contraction] Assume that \eqref{hyp_F} holds and $\eta>0$ sufficiently small. Let $\bar\rho_t$ and $\bar\gamma_t$ be two solutions of \eqref{intro_new_pde} with initial laws $\rho_0,\gamma_0\in \calP(\Pi)$. Then there exist constants $c>0$ and $C\ge1$, independent of $t$, such that for all $t\ge0$,
\[
\calW_1(\bar\rho_t,\bar\gamma_t)\le C \lt(\calW_1(\rho_0,\gamma_0) + \calW_1(\rho_0,\gamma_0)^\frac17 \rt) e^{-c t}.
\]
In particular, the nonlinear dynamics admits a unique stationary distribution $\rho_\infty$, and
\[
\calW_1(\bar\rho_t,\rho_\infty)\le C  \lt( \calW_1(\rho_0,\rho_\infty) +  \calW_1(\rho_0,\rho_\infty)^\frac17\rt)e^{-c t}.
\]
\end{maintheorem}

Finally, we investigate the complementary regime of stronger interactions. While Theorems \ref{thm_propa} and \ref{thm_lt} yield uniqueness and convergence to the homogeneous equilibrium in the weak-interaction regime, we show that for sufficiently large interaction strength the homogeneous stationary state may lose stability through a bifurcation mechanism. Such bifurcation phenomena for stationary solutions of single-species aggregation--diffusion equations were studied in \cite{CGPS}, we extend this analysis to the two-type system \eqref{pde}.

\begin{maintheorem}[Bifurcation criterion for large interaction strength]
  Assume that the interaction potentials $U$ and $V$ are coordinate-wise even. Then the system \eqref{pde} admits the spatially homogeneous stationary state
\[
(u_c,v_c)=\left(\frac{\alpha_{-1}}{\alpha_1+\alpha_{-1}},\, \frac{\alpha_1}{\alpha_1+\alpha_{-1}}\right) \quad\text{for all }\eta\ge0.
\]
Moreover, there exists an explicit threshold $\eta_*>0$, computable in terms of the Fourier modes of the interaction potentials $U$ and $V$, such that if $\eta$ crosses $\eta_*$ and a natural nondegeneracy condition holds, then $(u_c,v_c)$ becomes a bifurcation point. More precisely, a nontrivial branch of stationary solutions emerges from $(u_c,v_c)$ in a neighborhood of $\eta=\eta_*$.
\end{maintheorem}

%
%
%
%
%
%
\subsection{Organization of paper}

The rest of this paper is organized as follows. In Section \ref{sec_pre}, we collect the probabilistic ingredients used throughout the paper. We recall the $N$-particle diffusion--jump dynamics and record its forward Liouville equation, and we then review the nonlinear McKean--Vlasov limit together with the coupling tools. In particular, we introduce the one-dimensional spin-flip dynamics for the internal state and a reflection coupling for the spatial diffusion component. Section \ref{sec_poc} is devoted to establishing a uniform-in-time propagation of chaos estimate for the interacting particle system \eqref{eq0}--\eqref{eq1}. More precisely, we obtain a uniform-in-time control of the $1$-Wasserstein distance between the empirical measure $\mu_t^N$ and the mean-field law $\bar\rho_t$, using a mixed coupling argument that combines reflection coupling in space with an optimal coupling of the spin variables. Section \ref{sec_lt} turns to the mean-field dynamics and proves an exponential contraction estimate between two solutions of the McKean--Vlasov equation  \eqref{new_pde} in the $1$-Wasserstein distance. As a consequence, we obtain exponential convergence of any solution toward the unique stationary equilibrium distribution in the regime of small interaction strength. Finally, in Section \ref{sec_bif}, we study the stationary problem for large interaction strength and show that multiple equilibria may emerge through a bifurcation from the spatially homogeneous state. We derive a computable threshold $\eta_*$ and provide a bifurcation criterion based on the Crandall--Rabinowitz theorem.

%
%
%
%
%
%

\section{Preliminaries}\label{sec_pre}

%
%
%
%
%
%

\subsection{Finite-particle system and Liouville equation}

For later reference, we record the forward (Liouville) equation satisfied by the law $\rho_N(t)$ of the $N$-particle diffusion--jump system \eqref{eq0}--\eqref{eq1}.  Applying It\^o's formula to smooth test functions on $\Pi^N$ yields a linear evolution equation consisting of a diffusion operator in the spatial variables, a transport term generated by the mean-field interaction, and gain--loss terms encoding the independent type switches.  Its explicit form reads
\[\begin{aligned}
\pa_t \rho_N &=\Delta_N \rho_N -\frac1N\sum_{i,j=1}^N \nabla_{x_i}\cdot\lt(\rho_N\,F(x_i-x_j,y_j)\rt)  +\alpha_1\sum_{i=1}^N\lt(\chi_{-1}(y_i)\rho_N(\bx,\tilde\Theta_{-1}^i(\by)) -\chi_{1}(y_i)\rho_N(\bx,\by)\rt) \\
&\quad +\alpha_{-1}\sum_{i=1}^N\lt(\chi_{1}(y_i)\rho_N(\bx,\tilde\Theta_{1}^i(\by))-\chi_{-1}(y_i)\rho_N(\bx,\by)\rt),
\end{aligned}\]
where $\Delta_N$ denotes the diffusion operator acting on the spatial variables with the type-dependent coefficients inherited from \eqref{eq0}, given by
\[
\Delta_N \rho_N := \sum_{i=1}^N \frac{\sigma^2(y_i)}{2} \Delta_{x_i} \rho_N,
\]
and $\tilde\Theta_{\pm1}^i$ modifies the spin configuration by resetting the $i$-th component:
\[
\tilde \Theta_1^i(\by):=(y_1,\dots,y_{i-1},-1,y_{i+1},\dots,y_N),\quad \tilde \Theta_{-1}^i(\by):=(y_1,\dots,y_{i-1},1,y_{i+1},\dots,y_N).
\]
A derivation can be found in \cite[Section 2]{ACCL}.  We remark that the Liouville formulation is not used directly in the sequel; it is recalled here mainly to highlight that the particle system defines a well-posed Markov process whose law evolves according to a linear equation on $\Pi^N$.

%
%
%
%
%
%

\subsection{Nonlinear McKean--Vlasov process}

The nonlinear equation \eqref{intro_new_pde}, which governs the limiting density of the particle system \eqref{eq0}--\eqref{eq1}, was introduced in the Introduction. In this subsection we show that it admits an equivalent formulation in terms of a nonlinear McKean--Vlasov stochastic process.
Let $\bar\rho_t$ be a probability measure on $\Pi=\T^d\times\{\pm1\}$. We consider the nonlinear evolution equation
\bq\label{new_pde}
\pa_t\bar\rho = \frac{\sigma^2(y)}{2}\Delta_x\bar\rho - \nabla_x\cdot(\bar\rho\,F\oast\bar\rho) + {\rm T}\bar\rho,
\eq
which coincides with \eqref{intro_new_pde}, where ${\rm T}$ denotes the linear operator describing transitions between the two internal states with rates $\alpha_1$ and $\alpha_{-1}$.

To clarify the probabilistic structure, we first associate with a given measure-valued curve $(\bar\rho_t)_{t\ge0}$ a linear jump--diffusion process. Let $(\bar \rho_t)_{t\ge0}$ be a measurable curve of probability measures in $L^p(\Pi)$, $p \in (1,\infty)$, with initial data $\bar\rho_0$. We define
a $\Pi$-valued process $(\bar X_t,\bar Y_t)$ by 
\begin{align} 
\d\bar X_t &= \lt(\iint_\Pi F(\bar X_t-x,y)\,\d\bar\rho_t(x,y)\rt)\d t + \sigma(\bar Y_t)\,\d B_t, \label{main_lim}\\
\bar Y_t &= \bar Y_0+ \int_0^t ( -1 - \bar Y_{s-} )\, \d\tilde E_{1}(s) + \int_0^t ( 1 - \bar Y_{s-} )\, \d\tilde E_{-1}(s), \label{barY}
\end{align}
where
\[
  \tilde E_1(t) = E_1 \left( \al_1 \int_0^t \chi_1(  \bar Y_{s-}) \d s \right) \quad \mbox{and} \quad
    \tilde E_{-1}(t) = E_{-1} \left( \al_{-1} \int_0^t \chi_{-1}(  \bar Y_{s-}) \d s \right).
    \]
Here $E_1$ and $E_{-1}$ are the independent unit-rate Poisson processes and $\chi_k:\{1,-1\}\to\{0,1\}$, $k=\pm1$, is defined by $\chi_k(l)=1$ if $l=k$ and $\chi_k(l)=0$ otherwise. 

The following proposition makes precise the equivalence between the nonlinear PDE \eqref{new_pde} and the McKean--Vlasov process  \eqref{main_lim}--\eqref{barY}.

\begin{proposition}
Let $(\bar X_t,\bar Y_t)$ be a solution of \eqref{main_lim}--\eqref{barY} with initial law $\bar\rho_0$. Then, the law $\eta_t := \Law(\bar X_t,\bar Y_t)$ belongs to $C([0,\infty);L^p(\Pi))$ and is the unique distributional solution to 
\bq\label{eq_lin}
\pa_t \eta_t = \frac{\sigma^2(y)}{2}\Delta_x \eta_t - \nabla_x \cdot (\eta_t\, F \oast \bar\rho_t) + {\rm T} \eta_t,
\eq
with initial data $\bar\rho_0$, where the operator ${\rm T}$ is given by
\[
{\rm T}\eta(x,y) = (-1)y \alpha_1 \eta(x,1) + y \alpha_{-1} \eta(x,-1).
\]
In particular, if $(\bar\rho_t)_{t\ge0}$ is a solution of the nonlinear equation \eqref{new_pde}, then 
\[
\bar\rho_t=\Law(\bar X_t,\bar Y_t), \quad \text{for all $t\ge0$}.
\]
\end{proposition}
 
\begin{proof}
Let $\varphi\in C^2(\Pi)$ be an arbitrary test function. By the definition of the law,
\[
\EE \varphi(\bar X_t,\bar Y_t) = \iint_\Pi \varphi(z)\,\d\eta_t(z).
\]
 Applying It\^o's formula for jump--diffusion processes to
$\varphi(\bar X_t,\bar Y_t)$ yields
\begin{align*}
\varphi(\bar X_t, \bar Y_t) &= \varphi(\bar X_0, \bar Y_0) + \int_0^t \frac{\sigma^2(\bar Y_s)}{2} \Delta_x \varphi(\bar X_s, \bar Y_s)\,\d s   + \int_0^t \sigma(\bar Y_s)\nabla_x \varphi(\bar X_s, \bar Y_s)\cdot \d B_s \\
&\quad + \int_0^t \nabla_x \varphi(\bar X_s, \bar Y_s)\cdot \lt(\iint_\Pi F(\bar X_s-x,y)\,\d\bar\rho_s(x,y)\rt)\d s \\
&\quad + \int_0^t \lt\{ \varphi(\bar X_{s-},-1) - \varphi(\bar X_{s-},\bar Y_{s-}) \rt\}\,\d\tilde E_1(s) \\
&\quad + \int_0^t \lt\{ \varphi(\bar X_{s-},1) - \varphi(\bar X_{s-},\bar Y_{s-}) \rt\}\,\d\tilde E_{-1}(s) \\
&=: \sum_{i=1}^6 {\rm I}_i .
\end{align*}
The stochastic integral with respect to $\d B_s$ has zero expectation, hence
$\EE {\rm I}_3 = 0$. Taking expectations and using Fubini's theorem, we obtain
\begin{align*}
\EE {\rm I}_1 &= \iint_\Pi \varphi(z)\,\d\eta_0(z),\\
\EE {\rm I}_2 &= \int_0^t \iint_\Pi \frac{\sigma^2(y)}{2}\Delta_x \varphi(z)\,\d\eta_s(z)\,\d s,\\
\EE {\rm I}_4 &= \int_0^t \iint_\Pi \nabla_x \varphi(x,y)\cdot(F\oast\bar\rho_s)(x)\,\d\eta_s(x,y)\,\d s.
\end{align*}

For the jump terms, we use that $\tilde E_{\pm1}$ are time-changed Poisson processes with intensities $\alpha_{\pm1}\chi_{\pm1}(\bar Y_{s-})$. Therefore,
\begin{align*}
\EE {\rm I}_5 &= \int_0^t \iint_\Pi \alpha_1 \lt\{ \varphi(x,-1) - \varphi(x,y) \rt\} \chi_1(y)\,\d\eta_s(x,y)\,\d s,\\
\EE {\rm I}_6 &= \int_0^t \iint_\Pi \alpha_{-1} \lt\{ \varphi(x,1) - \varphi(x,y) \rt\} \chi_{-1}(y)\,\d\eta_s(x,y)\,\d s.
\end{align*}
Note that
\begin{align*}
\iint_\Pi {\rm T}\eta(z) \varphi(z)\,\d z &= \iint_\Pi \left\{ (-1)y \alpha_1 \eta(x,1) + y \alpha_{-1} \eta(x,-1)\right\} \varphi(x,y)\,\d z \cr
&= \int_{\T^d} \left\{ (-1) \alpha_1 \eta(x,1) + \alpha_{-1} \eta(x,-1) \right\} \varphi(x,1)\,\d x \cr
&\quad + \int_{\T^d} \left\{  \alpha_1 \eta(x,1) - \alpha_{-1} \eta(x,-1) \right\} \varphi(x,-1)\,\d x \cr
&= \int_{\T^d} \alpha_1 \lt\{ \varphi(x,-1) - \varphi(x,1) \rt\} \d\eta(x,1) \cr
&\quad + \int_{\T^d} \alpha_{-1} \lt\{ \varphi(x,1) - \varphi(x,-1) \rt\}  \d\eta(x,-1).
\end{align*}
Combining the above identities yields
\[
\frac{\d}{\d t}\iint_\Pi \varphi\,\d\eta_t = \iint_\Pi \lt( \frac{\sigma^2(y)}{2}\Delta_x \varphi + \nabla_x \varphi\cdot(F\oast\bar\rho_t) \rt) \d\eta_t
+ \iint_\Pi {\rm T}\eta_t\,\varphi\,\d z,
\]
which is precisely the weak formulation of \eqref{eq_lin}. Hence $\eta_t$ is a distributional solution of \eqref{eq_lin}.

Finally, since $F\in L^\infty(\Pi)$ and $\bar\rho_0\in L^p(\Pi)$ with $p\in(1,\infty)$, it follows from \cite[Theorem 2]{ACCL} that the linear equation \eqref{eq_lin} is well-posed in the class $C([0,\infty);L^p(\Pi))$. In particular, there exists a unique solution $\tilde\rho\in C([0,\infty);L^p(\Pi))$ to \eqref{eq_lin} with initial datum $\bar\rho_0$. On the other hand, we have shown above that $\eta_t=\Law(\bar X_t,\bar Y_t)$ is a distributional solution of \eqref{eq_lin} with the same initial condition $\eta_0=\bar\rho_0$. By uniqueness, we conclude that $\eta_t=\tilde\rho_t$ for all $t\ge0$. Setting $\bar\rho_t:=\tilde\rho_t$ yields
\[
\bar\rho_t=\Law(\bar X_t,\bar Y_t), \quad
t\mapsto\bar\rho_t\in C([0,\infty);L^p(\Pi)),
\]
which completes the proof.
\end{proof}
In \eqref{main_lim}--\eqref{barY}, the type component $\bar Y_t$ evolves autonomously and does not depend on the spatial variable $\bar X_t$. 
Let $\psi:\{-1,1\}\to\mathbb R$ and denote by $\xi_t:=\Law(\bar Y_t)$ its law.

Applying It\^o's formula for jump processes to $\psi(\bar Y_t)$ and taking expectations, we obtain
\begin{align*}
\frac{\d}{\d t}\int_{\{-1,1\}} \psi(y)\,\d\xi_t(y)
&= \alpha_{1}\big(\psi(-1)-\psi(1)\big)\,\xi_t(1)
   + \alpha_{-1}\big(\psi(1)-\psi(-1)\big)\,\xi_t(-1).
\end{align*}

Therefore, the Markov semigroup associated with $\bar Y_t$ has generator $L$ given by
\begin{equation}\label{bargen}
L f (y) = c(y)\big(f(-y)-f(y)\big),
\end{equation}
for $f:\{-1,1\}\to\mathbb R$, where
\[
c(y)=
\begin{cases}
\alpha_1, & y=1,\\
\alpha_{-1}, & y=-1.
\end{cases}
\]

%
%
%
%
%
%

 \subsection{One-dimensional spin-flip dynamics for the internal state} \label{section_flip}

In this part, we recall some basic facts on spin-flip Markov processes, which provide an equivalent microscopic description of the internal state dynamics appearing in both the finite-particle system and the nonlinear McKean--Vlasov limit. We follow standard notation from the theory of interacting particle systems and refer to \cite{Li,St} for further background.

Let $\Lambda_N := \{1,\dots,N\}$ and denote by $\Sigma_N := \{-1,+1\}^{\Lambda_N}$ the space of spin configurations. An element $\sigma\in\Sigma_N$ is written as $\sigma=(\sigma(i))_{i\in\Lambda_N}$, where $\sigma(i)\in\{-1,+1\}$ represents the internal state of particle $i$.
For $j\in\Lambda_N$, we denote by $\sigma^j$ the configuration obtained from $\sigma$ by
flipping the spin at site $j$, i.e.
\[
\sigma^j(k)=
\begin{cases}
-\sigma(j), & k=j,\\
\sigma(k), & k\neq j.
\end{cases}
\]
The spin-flip dynamics is defined through site-dependent flip rates $c(\sigma,j)$. The generator of the associated Markov process on $\Sigma_N$ is given by
\[
L_N f(\sigma) = \sum_{j\in\Lambda_N} c(\sigma,j)\lt(f(\sigma^j)-f(\sigma)\rt),
\]
acting on functions $f:\Sigma_N\to\R$. We denote by $(\sigma_t)_{t\ge0}$ the Markov process generated by $L_N$. Equivalently, $\sigma_t$ solves the martingale problem associated with $L_N$, namely
\bq\label{Mar}
f(\sigma_t)-f(\sigma_0)-\int_0^t L_N f(\sigma_s)\,\d s = M_t,
\eq
where $(M_t)_{t\ge0}$ is a martingale with respect to the natural filtration.

In the present work, we consider the constant flip rates
\bq\label{jump_rate}
c(\sigma,i)=
\begin{cases}
\alpha_1, & \sigma(i)=1,\\
\alpha_{-1}, & \sigma(i)=-1,
\end{cases}
\eq
which correspond exactly to the transition rates appearing in the jump SDE representation \eqref{eq1}. With this choice, each particle changes its internal state independently of the others, and the spin dynamics is spatially homogeneous.

To illustrate the resulting law of the internal state, consider the coordinate function $f^i:\Sigma_N\to\R$ defined by $f^i(\sigma)=\sigma(i)$.
A direct computation yields
\bq \label{igen}
L_N f^i(\sigma) = c(\sigma,i)\lt(\sigma^i(i)-\sigma(i)\rt) = -2 c(\sigma,i) \sigma(i).
\eq
Taking expectations in \eqref{Mar} with $f^i$ and using \eqref{jump_rate}, we obtain
\[
\mathbb{P}(\sigma_t(i)=1)-\mathbb{P}(\sigma_t(i)=-1) = \mathbb{E}[\sigma_0(i)] +2\int_0^t
-\alpha_1\,\mathbb{P}(\sigma_s(i)=1) +\alpha_{-1}\,\mathbb{P}(\sigma_s(i)=-1)\d s.
\]
Since $\mathbb{P}(\sigma_t(i)=1)+\mathbb{P}(\sigma_t(i)=-1)=1$, this identity can be rewritten as the closed system
\[\begin{aligned}
\mathbb{P}(\sigma_t(i)=1)
&= \mathbb{P}(\sigma_0(i)=1) - \int_0^t \alpha_1\,\mathbb{P}(\sigma_s(i)=1) -\alpha_{-1}\,\mathbb{P}(\sigma_s(i)=-1)\d s,\\ 
\mathbb{P}(\sigma_t(i)=-1) &= \mathbb{P}(\sigma_0(i)=-1)+ \int_0^t \alpha_1\,\mathbb{P}(\sigma_s(i)=1)-\alpha_{-1}\,\mathbb{P}(\sigma_s(i)=-1) \d s.
\end{aligned}\]
Solving this linear ODE system yields the explicit formula
\begin{align}\label{spin_law_0}
\begin{aligned}
\mathbb{P}(\sigma_t(i)=1) &= e^{-(\alpha_1+\alpha_{-1})t}\,\mathbb{P}(\sigma_0(i)=1) + \frac{\alpha_{-1}}{\alpha_1+\alpha_{-1}} \lt(1-e^{-(\alpha_1+\alpha_{-1})t}\rt),\\
\mathbb{P}(\sigma_t(i)=-1) &= e^{-(\alpha_1+\alpha_{-1})t}\,\mathbb{P}(\sigma_0(i)=-1) + \frac{\alpha_1}{\alpha_1+\alpha_{-1}}
\lt(1-e^{-(\alpha_1+\alpha_{-1})t}\rt).
\end{aligned}
\end{align}

In particular, the law of $\sigma_t(i)$ is independent of $i$, reflecting the fact that each particle undergoes the same autonomous spin-flip dynamics. The exponential relaxation rate $\alpha_1+\alpha_{-1}$ will play a key role in the coupling arguments developed later, both at the microscopic and at the mean-field level.

%
%
%
%
%
%

\subsection{Reflection coupling}\label{reflection coupling}
 The main analytical tool in this paper is a coupling construction tailored to interacting diffusion--jump systems with internal state switching. Our strategy combines a reflection-type coupling for the spatial diffusion component with an optimal coupling of the spin-flip dynamics governing the internal states. This mixed coupling is specifically designed to be compatible with the Wasserstein distance introduced earlier and forms the backbone of the proofs of uniform-in-time propagation of chaos and long-time contraction.

A reflection coupling was originally introduced for multidimensional diffusion processes in \cite{LR} and has since become a standard and robust tool for establishing quantitative stability properties, including exponential convergence to equilibrium \cite{Eb}. More recently, it has been successfully employed to obtain uniform-in-time propagation of chaos for Vlasov--McKean equations with weak interactions; see, for instance, \cite{DAGR}. In the present setting, we adapt this methodology to a system in which particles undergo both diffusive motion in space and random transitions between two internal states.

Before describing the full coupling construction for the interacting particle system, we briefly recall the classical reflection coupling for diffusion processes and then explain how it can be combined with an optimal coupling of the spin-flip dynamics in our framework.

Consider a diffusion process $(X_t)_{t\ge0}$ in $\R^d$ solving
\bq\label{sde_0}
\d X_t = b(X_t)\,\d t + \sigma\,\d B_t,
\eq
where $B_t$ is a $d$-dimensional Brownian motion, $\sigma$ is a constant $d\times d$ matrix with $\det\sigma>0$, and $b:\R^d \to \R^d$ is locally Lipschitz. A reflection coupling of two solutions of \eqref{sde_0} with initial laws $\mu$ and $\nu$ is a diffusion process $(X_t,\hat X_t)_{t\ge0}$ with initial law $\mathcal{L}(X_0,\hat X_0)=\gamma$, where $\gamma$ is a coupling of $\mu$ and $\nu$, and such that
\[
\begin{aligned}
\d X_t &= b(X_t)\,\d t + \sigma\,\d B_t,\\
\d \hat X_t &= b(\hat X_t)\,\d t + \sigma\bigl(I_d-2 e_t e_t^{\top}\bigr)\,\d B_t, \quad t<T,\\
\hat X_t &= X_t, \quad t\ge T,
\end{aligned}
\]
where
\[
T:=\inf\{t\ge0: X_t=\hat X_t\}, \quad e_t:=\frac{\sigma^{-1}(X_t-\hat X_t)}{|\sigma^{-1}(X_t-\hat X_t)|}.
\]
Here the matrix $I_d-2e_t e_t^{\top}$ represents the reflection with respect to the hyperplane
\[
H_{e_t}:=\{x\in\R^d:\langle e_t,x\rangle=0\}.
\]
Equivalently, the reflection transform acts on vectors $x\in\R^d$ as
\[
R_{e_t}(x)=x-2\langle e_t,x\rangle e_t,
\]
that is, it leaves the component orthogonal to $e_t$ unchanged while reversing the component in the direction of $e_t$. The reflected noise $(I_d-2e_t e_t^{\top})\,\d B_t$ remains a Brownian motion by L\'evy's characterization.
 
We now define the interaction particle system $(X_t^{i,N},Y_t^{i,N})_{i=1}^N$ approximating $\rho_t$ and explain the coupling  with $N$ independent copies $(\bar X_t^i,\bar Y_t^i)_{i=1}^N$ of the nonlinear McKean--Vlasov process given by \eqref{main_lim}--\eqref{barY}. The internal state processes $\mathbf Y_t=(Y_t^{i,N})_{i=1}^N$ form  the  spin-flip processes with rates $\alpha_1$ and $\alpha_{-1}$, introduced in Section \ref{section_flip} with $\mathbf Y_t$ with $\mathbf Y_t$ identified with $\sigma_t$.
As shown in \eqref{spin_law_0}, their one-particle marginals satisfy
\begin{align}\label{spin_law_Y}
\begin{aligned}
\mathbb{P}(Y_t^{i,N}=1) &= e^{-(\alpha_1+\alpha_{-1})t}\,\mathbb{P}(Y_0^{i,N}=1) + \frac{\alpha_{-1}}{\alpha_1+\alpha_{-1}} \lt(1-e^{-(\alpha_1+\alpha_{-1})t}\rt),\\
\mathbb{P}(Y_t^{i,N}=-1) &= e^{-(\alpha_1+\alpha_{-1})t}\,\mathbb{P}(Y_0^{i,N}=-1) + \frac{\alpha_{1}}{\alpha_1+\alpha_{-1}} \lt(1-e^{-(\alpha_1+\alpha_{-1})t}\rt)
\end{aligned}
\end{align}
for all $i$.
From \eqref{bargen} and \eqref{igen} we notice that $\bar Y_t$ shares the same generator $L$ as the marginal dynamics of $Y_t^{i,N}$.
 Indeed, choosing the identity as the test function, the martingale equation associated with $\bar Y_t$ is given by
 $$ \bar Y_t - \bar Y_0 - \int_0^t L \bar Y_s \, ds = M_t, $$where $L \bar Y_s = -2c(\bar Y_s)\bar Y_s$ according to \eqref{bargen}. 
 Thus, the evolution of $\EE[\bar Y_t]$ is governed by the same linear system as that of $\EE[Y_t^{i,N}]$ in \eqref{spin_law_Y}.
 To compare the two processes, we choose the joint law of $(Y_t^{i,N}, \bar Y_t^i)$ through the optimal coupling that attains the $\calW_1$ distance:
 $$ \EE |Y_t^{i,N} - \bar Y_t^i| = \calW_1\bigl(\Law(Y_t^{i,N}), \Law(\bar Y_t^i)\bigr). $$
 Since the state space is the finite set $\{-1, 1\}$, this optimal coupling is explicitly identified,
  and the distance can be expressed as
  \[
  \calW_1\bigl(\Law(Y_t^{i,N}), \Law(\bar Y_t^i)\bigr) = 2 \bigl|\mathbb{P}(Y_t^{i,N}=1) - \mathbb{P}(\bar Y_t^i=1)\bigr|.
 \]
Here we used that for $y,\bar y\in\{\pm1\}$,
\[
|y-\bar y| = 2 \mathbf 1_{\{y\neq \bar y\}}.
\]

To couple the spatial components $(X^{i, N}_t, \bar X^i_t)$, we introduce a smooth partition of unity $\phi_r^\delta,\phi_s^\delta:\R^d\to[0,1]$ satisfying
\begin{align}\label{separation}
(\phi_r^\delta)^2+(\phi_s^\delta)^2=1, \quad \phi_r^\delta(x)=
\begin{cases}
1, & |x|\ge\delta,\\
0, & |x|\le\delta/2,
\end{cases}
\end{align}
for some $\delta>0$. Setting ${\rm D}_t^i:=\bar X_t^i-X_t^{i,N}$ and $e_t^i=n({\rm D}_t^i)$ with $n(x)=x/|x|$ for $x\neq0$, we define the coupled dynamics by
\begin{align*}
\d X^{i, N}_t  &=   \phi_r^\del ({\rm D}_t^i)(I_d -2 e^i_t(e^i_t)^T)  \sigma(  Y^i_t) \, \d B^i_t +    \phi_s^\del ({\rm D}_t^i)  \sigma( Y^i_t)\, \d\tilde B^i_t + \frac 1N \sum_{j=1}^N  F(\xit-\xjt, Y^j_t)\, \d t,\\
 \d \bar X^i_t&  =    \phi_r^\del ({\rm D}_t^i)  \sigma( \bar Y^i_t) \, \d B^i_t+ \phi_s^\del ({\rm D}_t^i)  \sigma( \bar Y^i_t) \, \d\tilde B^i_t  +\iint_\Pi F(\bar X^i_t-x, y) \, \d\bar \rho_t(z) \d t \quad \mbox{ with }\mathrm{ Law }(\bar X^i_t, \bar Y^i_t) = \bar\rho_t,
\end{align*}
where $(B_t^i,\tilde B_t^i)_{i=1}^N$ are independent Brownian motions.

Throughout the sequel, we denote by $(X_t^{i,N},Y_t^{i,N})_{i=1}^N$ the interacting particle system and by $(\bar X_t^i,\bar Y_t^i)_{i=1}^N$ independent copies of the nonlinear McKean--Vlasov process. When no confusion arises, we drop the superscript $N$ in $X_t^{i,N}$ and $Y_t^{i,N}$.

%
%
%
%
%
%

\section{Uniform-in-time propagation of chaos}\label{sec_poc}
 
In this section, we prove the uniform-in-time propagation of chaos for the stochastic particle system \eqref{eq0}--\eqref{eq1}. 

Let $(X_t^i,Y_t^i)_{i=1}^N$ be a solution of the interacting particle system
\eqref{eq0}--\eqref{eq1}, and  recall that the associated empirical measure is the $\calP(\Pi)$-valued random variable  given by
\[
\mu_t^N = \frac1N \sum_{i=1}^N \delta_{(X_t^i,Y_t^i)} \in \calP(\Pi),
\quad \Pi=\T^d\times\{\pm1\}.
\]
Let $(\bar X_t,\bar Y_t)$ denote the limiting nonlinear process solving the McKean--Vlasov SDE associated with \eqref{new_pde}, and let 
\[
\bar\rho_t := {\rm Law}(\bar X_t,\bar Y_t)
\]
be its time-marginal distribution. The main result of this section provides a uniform-in-time estimate for the $1 $-Wasserstein distance between the empirical measure $\mu_t^N$ and its mean-field limit $\bar\rho_t$.

\begin{theorem}\label{thm_propa}  
Assume that \eqref{hyp_F} holds. Then for sufficiently small $\eta>0$ so that 
\[
c_0 := \frac{2\pi}{d} ( \sigma_{\rm min}^2 \pi - \eta \sqrt{d}) > 0.
\]
Then there exists a constant $C>0$, independent of $N$ and $t$, such that for all  $t >0$ and  $N \geq 2$,
\begin{align}\label{W1x}
\begin{aligned}
\mathbb E \calW_1(\pi_x {}_\# \mu^N_t,  \pi_x {}_\# \bar \rho_t) &\le  C   e^{-c_0t} \mathbb E \calW_1(\pi_x {}_\# \mu^N_0,  \pi_x {}_\# \bar \rho_0) +   \frac C{\sqrt{N}} (1 + {\bf 1}_{\{d=2\}} \log N )    \cr
&\quad   + CN^{\frac52}e^{-\min\{c_0, \,(\al_1+\al_{-1})\}t} \mathbb E \calW_1(\pi_y {}_\# \mu^N_0,  \pi_y {}_\# \bar \rho_0)
\end{aligned}
\end{align}
and
\[
\mathbb E \calW_1(\pi_y {}_\# \mu^N_t,  \pi_y {}_\# \bar \rho_t) \leq e^{-(\al_1+\al_{-1})t}\mathbb E \calW_1(\pi_y {}_\# \mu^N_0,  \pi_y {}_\# \bar \rho_0).
\]
\end{theorem}

\begin{remark} When the initial type distribution $\pi_y{}_\#\mu^N_0$ coincides with that of $\pi_y{}_\#\bar\rho_0$, the type variable $Y_t^i$ and its mean-field counterpart $\bar Y_t$ are perfectly synchronized at $t=0$. In this case, the internal-state distance vanishes,
\[
\mathbb E \calW_1(\pi_y {}_\# \mu^N_0,  \pi_y {}_\# \bar \rho_0) = 0,
\]
and the last term in \eqref{W1x} disappears. Consequently, the $x$-marginal satisfies a closed propagation of chaos estimate with the rate $O(N^{-1/2})$, uniformly in time.

In contrast, if the initial type distribution is not matched $($that is, $Y_0^i$ and $\bar Y_0^i$ are independent or differently biased $)$, the discrepancy in the type variable indirectly perturbs the spatial dynamics through the mean-field interaction $F(\cdot,Y_t^j)$. This effect is precisely quantified by the term
\[
N^{\frac52}e^{-\min\{\frac{2\pi}{d} ( \sigma_{\rm min}^2 \pi - \eta \sqrt{d}), \,(\al_1+\al_{-1})\}t} \mathbb E \calW_1(\pi_y {}_\# \mu^N_0,  \pi_y {}_\# \bar \rho_0).
\]
The algebraic prefactor $N^{\frac52}$, induced by the jump mechanism, is counteracted by an exponential decay with rate $\min\{c_0, \,(\alpha_1+\alpha_{-1})\}$. It is worth noting that the temporal decay persists even in the absence of the jump term due to the diffusive factor; however, the effective convergence rate is determined by the slower of the two. Consequently, after a logarithmic burn-in time $t\gtrsim \log N$, this transient contribution becomes negligible and the system recovers the uniform-in-time $N^{-1/2}$ rate for the $x$-marginal.  

Therefore, the propagation of chaos in this two-type interacting system depends sensitively on the initial alignment of the type distribution: perfect type synchronization ensures instantaneous decoupling between the $x$- and $y$-dynamics, while type heterogeneity generates a transient but quantifiable delay in achieving the uniform rate.
\end{remark}

%
%
%
%
%
%

  \subsection{Coupling estimate}\label{sec_propa}

We estimate the Wasserstein distance $\EE\calW_{1}(\pi_{x\#}\mu_{t}^{N}, \pi_{x\#}\overline{\rho}_{t})$
via a coupling argument. Let $(\bar X_t,\bar Y_t)$ be the nonlinear McKean--Vlasov process and let
$(\bar X_t^i,\bar Y_t^i)_{i=1}^N$ be i.i.d.\ copies with common law $\bar\rho_t$. Define the associated
empirical measure
\[
\overline{\mu}_{t}^{N}
:= \frac{1}{N}\sum_{i=1}^{N}\delta_{(\overline{X}_{t}^{i}, \overline{Y}_{t}^{i})}.
\]
By the triangle inequality,
\[
 \calW_{1}(\pi_{x\#}\mu_{t}^{N}, \pi_{x\#}\overline{\rho}_{t})
\le
 \calW_{1}(\pi_{x\#}\mu_{t}^{N}, \pi_{x\#}\overline{\mu}_{t}^{N})
+
 \calW_{1}(\pi_{x\#}\overline{\mu}_{t}^{N}, \pi_{x\#}\overline{\rho}_{t}).
\]
 The second term on the right-hand side can be estimated by the classical empirical measure convergence result for the Wasserstein distance. More precisely, by \cite[Theorem 1]{FG},
\[
\mathbb E \calW_{1}(\pi_{x\#}\overline{\mu}_{t}^{N}, \pi_{x\#}\overline{\rho}_{t}) \le C \left\{ \begin{array}{ll}
N^{-\frac12} & \textrm{if $d=1$},\\
N^{-\frac12} \log (1+N) & \textrm{if $d=2$},\\
N^{-\frac1d} & \textrm{if $d\ge3$}.
  \end{array} \right.
  \]
Hence, it suffices to bound $\EE \calW_{1}(\pi_{x\#}\mu_{t}^{N}, \pi_{x\#}\overline{\mu}_{t}^{N})$.

We fix a coupling between $(X_t^i,Y_t^i)$ and $(\overline X_t^i,\overline Y_t^i)$ as
constructed in Section \ref{reflection coupling}, and define
\[
{\rm D}_t^i := \overline X_t^i - X_t^i , \quad i=1,\dots,N.
\]
By the definition of the $1$-Wasserstein distance as an infimum over couplings, this
construction yields
\begin{equation}\label{eq:coupling_bound}
\EE \calW_{1}(\pi_{x\#}\mu_{t}^{N}, \pi_{x\#}\overline{\mu}_{t}^{N})
\le \frac{1}{N}\sum_{i=1}^{N} \EE|X_t^i - \overline{X}_t^i|
= \frac{1}{N}\sum_{i=1}^{N} \EE|{\rm D}_{t}^{i}|.
\end{equation}
Consequently, it suffices to establish a uniform-in-time estimate for the averaged
$L^{1}$-distance on the right-hand side of \eqref{eq:coupling_bound}.

To this end, we first compute the It\^o differential satisfied by $|{\rm D}^i_t|^2$ under the coupled dynamics; this identity will be used repeatedly in the remainder of the proof.

   \begin{lemma}\label{ito_E}Under the assumptions of Theorem \ref{thm_propa}, almost surely, for all $t \ge 0$ and
$i=1,\dots,N$, the following identity holds:

\begin{align*}
  \d | {\rm D}_t^i|^2 &= 2\langle {\rm D}_t^i, (F \oast \mu^N_t (X^i_t) -  F \oast \bar\rho_t (\bar X^i_t) )\rangle\, \d t  + d(\sigma(  Y^i_t) -  \sigma( \bar Y^i_t) )^2\,  \d t+ 4(\phi_r^\del)^2 ({\rm D}_t^i)   \sigma(  Y^i_t) \sigma( \bar Y^i_t) \, \d t \cr
&\quad + 2\langle {\rm D}_t^i,A^i_t \, \d B^i_t + \tilde A^i_t \, \d\tilde B^i_t \rangle,
\end{align*}
where
\[
\mu^N_t = \frac1N \sum_{j=1}^N \delta_{(X^j_t, Y^j_t)}, \quad A^i_t :=  \phi_r^\del ({\rm D}_t^i) \{ \sigma(  Y^i_t) - \sigma( \bar Y^i_t)   \} I_d  - 2   \phi_r^\del ({\rm D}_t^i) e^i_t \otimes e^i_t  \sigma(  Y^i_t),
\]
and
\[
\tilde A^i_t := \phi_s^\del ({\rm D}_t^i)  \{ \sigma(  Y^i_t) -  \sigma( \bar Y^i_t)   \} I_d.
\]
  \end{lemma}
  \begin{proof}
 Let us recall 
\[
\d X^{i}_t  =  \{ \phi_r^\del ({\rm D}_t^i)(I_d -2 e^i_t \otimes e^i_t )  \sigma(  Y^i_t) \, \d B^i_t +    \phi_s^\del ({\rm D}_t^i)  \sigma( Y^i_t) \, \d\tilde B^i_t\}   + F \oast \mu^N_t (X^i_t)\, \d t  
\]
and
  \[
 \d \bar X^i_t  =   \{ \phi_r^\del ({\rm D}_t^i)  \sigma( \bar Y^i_t) \, \d B^i_t+ \phi_s^\del ({\rm D}_t^i)  \sigma( \bar Y^i_t)\, \d\tilde B^i_t\}  + F \oast \bar\rho_t (\bar X^i_t) \,\d t
\]
(see \eqref{oast} for the notation $\oast$).\\

We now set
\begin{align*}
A^i_t &:= \phi_r^\del ({\rm D}_t^i)(I_d - 2 e^i_t \otimes e^i_t )  \sigma(  Y^i_t) -      \phi_r^\del ({\rm D}_t^i)  \sigma( \bar Y^i_t) =  \phi_r^\del ({\rm D}_t^i) \{ \sigma(  Y^i_t) - \sigma( \bar Y^i_t)   \} I_d - 2   \phi_r^\del ({\rm D}_t^i) e^i_t \otimes e^i_t  \sigma(  Y^i_t)
\end{align*}
and
\[
\tilde A^i_t := \phi_s^\del ({\rm D}_t^i)  \{ \sigma(  Y^i_t) -  \sigma( \bar Y^i_t)   \} I_d.
\]
Then we obtain
\[
\d {\rm D}_t^i = (F \oast \mu^N_t (X^i_t) -  F \oast \bar\rho_t (\bar X^i_t) )\,  \d t + A^i_t \, \d B^i_t + \tilde A^i_t \,  \d\tilde B^i_t.
\]
 Applying It\^o's lemma yields
\begin{align}\label{eq_dd}
\begin{aligned}
\frac12 \d | {\rm D}_t^i|^2 &= \langle {\rm D}_t^i, (F \oast \mu^N_t (X^i_t) -  F \oast \bar\rho_t (\bar X^i_t) )\rangle \, \d t  + \frac12 {\rm Tr}[ (A^i_t)^T A^i_t + (\tilde A^i_t)^T\tilde A^i_t ] \, \d t \cr
&\quad +\langle {\rm D}_t^i,A^i_t\, \d B^i_t + \tilde A^i_t \, \d\tilde B^i_t \rangle.
\end{aligned}
\end{align}
A direct computation gives
\[
(A^i_t)^T A^i_t = (A^i_t)^2 = (\phi_r^\del)^2 ({\rm D}_t^i) \left\{ (\sigma(  Y^i_t) - \sigma( \bar Y^i_t))^2 I_d  +  4 (e^i_t \otimes e^i_t)  \sigma(  Y^i_t) \sigma( \bar Y^i_t) \right\}
\]
and
\[
(\tilde A^i_t)^T\tilde A^i_t = (\tilde A^i_t)^2 = (\phi_s^\del)^2 ({\rm D}_t^i)(\sigma(  Y^i_t) -  \sigma( \bar Y^i_t) )^2 I_d. 
\]
Thus,
\begin{align}\label{ATA}
(A^i_t)^T A^i_t + (\tilde A^i_t)^T\tilde A^i_t &= (\sigma(  Y^i_t) -  \sigma( \bar Y^i_t) )^2 I_d + 4(\phi_r^\del)^2 ({\rm D}_t^i) (e^i_t \otimes e^i_t ) \sigma(  Y^i_t) \sigma( \bar Y^i_t)   
\end{align}
and
\[
 {\rm Tr}[ (A^i_t)^T A^i_t + (\tilde A^i_t)^T\tilde A^i_t ] = d(\sigma(  Y^i_t)  - \sigma( \bar Y^i_t) )^2+ 4(\phi_r^\del)^2 ({\rm D}_t^i)   \sigma(  Y^i_t) \sigma( \bar Y^i_t) .
\]
This implies that the second term on the right-hand side of \eqref{eq_dd} can be estimated as
\[
\frac12 {\rm Tr}[ (A^i_t)^T A^i_t + (\tilde A^i_t)^T\tilde A^i_t ]\, \d t =\frac d2(\sigma(  Y^i_t) - \sigma( \bar Y^i_t) )^2\,  \d t  + 2(\phi_r^\del)^2 ({\rm D}_t^i)   \sigma(  Y^i_t)  \sigma( \bar Y^i_t) \, \d t. 
\]
Hence, we arrive at
\begin{align*}
  \d | {\rm D}_t^i|^2 &= 2\langle {\rm D}_t^i, (F \oast \mu^N_t (X^i_t) -  F \oast \bar\rho_t (\bar X^i_t) )\rangle\, \d t  + d(\sigma(  Y^i_t) -  \sigma( \bar Y^i_t) )^2\,  \d t+ 4(\phi_r^\del)^2 ({\rm D}_t^i)   \sigma(  Y^i_t) \sigma( \bar Y^i_t) \, \d t \cr
&\quad + 2\langle {\rm D}_t^i,A^i_t \, \d B^i_t + \tilde A^i_t  \, \d\tilde B^i_t \rangle,
\end{align*}
and this completes the proof.
\end{proof}

%
%
%
%
%
%

\subsection{Proof of Theorem \ref{thm_propa}} We now derive a differential inequality for a suitable regularized distance along the coupling, leading to the uniform-in-time estimate in Theorem \ref{thm_propa}.

To handle the singularity of the distance function near zero, we introduce a regularization parameter $a$ following \cite{DAGR}.  Together with the separation parameter $\delta$ defined in \eqref{separation}, we thus work with two regularization parameters. They are used to control the additional terms generated by the type dynamics $(Y_t^i,\bar Y_t^i)$, and will be chosen as fractional powers of $N^{-1}$ at the end.

Define the auxiliary function
\[
\psi_a(r):=\sqrt{r+a}, \quad r\ge 0.
\]
Applying It\^o's formula to $\psi_a(|{\rm D}_t^i|^2)$ and using Lemma \ref{ito_E}, we obtain
\begin{align*}
\d \psi_a( | {\rm D}_t^i|^2) &= \psi'_a(| {\rm D}_t^i|^2) \, \d | {\rm D}_t^i|^2 + 2 \psi''_a(| {\rm D}_t^i|^2) \big\{ |(A^i_t)^T {\rm D}^i_t|^2 \, \d t + |(\tilde A^i_t)^T {\rm D}^i_t|^2 \, \d t\big\} \cr
&= 2\psi'_a(| {\rm D}_t^i|^2) \langle  {\rm D}_t^i, (F \oast \mu^N_t (X^i_t) -  F \oast \bar\rho_t (\bar X^i_t) )\rangle \, \d t  \cr
  &\quad + d \psi'_a(| {\rm D}_t^i|^2) (\sigma(  Y^i_t) -  \sigma( \bar Y^i_t) )^2 \, \d t+ 4 \psi'_a(| {\rm D}_t^i|^2) (\phi_r^\del)^2 ({\rm D}_t^i)   \sigma(  Y^i_t) \sigma( \bar Y^i_t)  \, \d t \cr
&\quad + 2\psi'_a(| {\rm D}_t^i|^2) \langle {\rm D}_t^i,A^i_t \, \d B^i_t + \tilde A^i_t \, \d\tilde B^i_t \rangle  +  2 \psi''_a(| {\rm D}_t^i|^2) \big\{ |(A^i_t)^T {\rm D}^i_t|^2 \d t + |(\tilde A^i_t)^T {\rm D}^i_t|^2 \,\d t\big\} .
\end{align*}
Using \eqref{ATA}, we have
\[
|(A^i_t)^T {\rm D}^i_t|^2 + |(\tilde A^i_t)^T {\rm D}^i_t|^2 \le |{\rm D}^i_t|^2(\sigma(  Y^i_t) -  \sigma( \bar Y^i_t))^2 + 4|{\rm D}^i_t|^2 (\phi_r^\del)^2  ({\rm D}_t^i)\sigma(  Y^i_t) \sigma( \bar Y^i_t).
\]
This together with
\[
2\psi_a''(r) r = - \psi'_a(r) -2a \psi''(r),
\]
yields
\begin{align*}
\d \psi_a( | {\rm D}_t^i|^2) &= 2\psi'_a(| {\rm D}_t^i|^2) \langle  {\rm D}_t^i, (F \oast \mu^N_t (X^i_t) -  F \oast \bar\rho_t (\bar X^i_t) )\rangle \, \d t   + (d-1) \psi'_a(| {\rm D}_t^i|^2) (\sigma(  Y^i_t) -  \sigma( \bar Y^i_t) )^2 \, \d t \cr
&\quad + 2\psi'_a(| {\rm D}_t^i|^2) \langle {\rm D}_t^i,A^i_t \, \d B^i_t + \tilde A^i_t  \, \d\tilde B^i_t \rangle \cr
&\quad -  2a \psi''_a(| {\rm D}_t^i|^2) \big\{  (\sigma(  Y^i_t) -  \sigma( \bar Y^i_t))^2 + 4  (\phi_r^\del)^2  ({\rm D}_t^i)\sigma(  Y^i_t) \sigma( \bar Y^i_t)\big\}\,  \d t .
\end{align*}

Recall that $f(r) = \sin (\frac{\pi}{\sqrt{d}}r)$ satisfies
\[
f''(r) = - \frac{\pi^2}{d}f(r) \quad \text{and} \quad f(r) \leq \frac{\pi}{\sqrt{d}} r \quad \text{for } r \ge 0.
\]
Applying It\^o's formula once more to $f(\psi_a(|{\rm D}_t^i|^2))$, we obtain 
\begin{align*}
 \d f( \psi_a( | {\rm D}_t^i|^2))  
&= f'(\psi_a( | {\rm D}_t^i|^2)) \, \d \psi_a( | {\rm D}_t^i|^2)   + 2f''(\psi_a( | {\rm D}_t^i|^2))(\psi'_a(| {\rm D}_t^i|^2))^2 \big\{ |(A^i_t)^T {\rm D}^i_t|^2 \, \d t + |(\tilde A^i_t)^T {\rm D}^i_t|^2 \, \d t\big\}   \cr
&= 2 f'(\psi_a( | {\rm D}_t^i|^2)) \psi'_a(| {\rm D}_t^i|^2) \langle  {\rm D}_t^i, (F \oast \mu^N_t (X^i_t) -  F \oast \bar\rho_t (\bar X^i_t) )\rangle\, \d t  \cr
&\quad + 2f'(\psi_a( | {\rm D}_t^i|^2)) \psi'_a(| {\rm D}_t^i|^2) \langle {\rm D}_t^i,A^i_t \, \d B^i_t + \tilde A^i_t  \, \d\tilde B^i_t \rangle \cr
&\quad + (d-1)f'(\psi_a( | {\rm D}_t^i|^2)) \psi'_a(| {\rm D}_t^i|^2) (\sigma(  Y^i_t) -  \sigma( \bar Y^i_t) )^2 \, \d t \cr
&\quad   -  2af'(\psi_a( | {\rm D}_t^i|^2)) \psi''_a(| {\rm D}_t^i|^2) \big\{  (\sigma(  Y^i_t) -  \sigma( \bar Y^i_t))^2 + 4  (\phi_r^\del)^2  ({\rm D}_t^i)\sigma(  Y^i_t) \sigma( \bar Y^i_t)\big\} \, \d t\cr
&\quad + 2f''(\psi_a( | {\rm D}_t^i|^2))(\psi'_a(| {\rm D}_t^i|^2))^2 \big\{ |(A^i_t)^T {\rm D}^i_t|^2\, \d t + |(\tilde A^i_t)^T {\rm D}^i_t|^2 \, \d t\big\}  \cr
&=: {\rm I}^i_t +\rm{ II}^i_t+ \rm{ III}^i_t + \rm{ IV}^i_t + V^i_t.
\end{align*}

\noindent $\bullet$ Estimate of ${\rm I}^i_t$: By Lemma \ref{ito_E}, we can estimate 
\[
\left|\left\langle e_t^i, (F \oast \mu^N_t (X^i_t) -  F \oast \bar\rho_t (\bar X^i_t) )\right\rangle\right| \leq \frac{1}{N} \sum_{j=1}^{N} |\Om^i_j| +| \Upsilon^i_t|,
\]
where
\begin{align*}
    \Om^i_j &=   F (X^i_t - X^j_t, Y^j_t)-F (\bar X^{i}_t - \bar X^{j}_t, Y^{j}_t) , \quad
    \Upsilon^i_t  = \frac{1}{N} \sum_{j=1}^{N}  F (\bar X^i_t - \bar X^j_t, Y^j_t)- \iint_\Pi F (\bar X^i_t -x, y) \,\d \bar \rho_t(z). 
\end{align*}
By the assumption on $F$ \eqref{hyp_F}, for all $i, j \in \{1, \dots, N\}$ and $t \ge 0$ we have
\begin{equation}\label{Omega_trigonometric}
    |\Omega^i_j|
    \leq \eta\, f(|{\rm D}^i_t - {\rm D}^j_t|)
    \leq \eta \big( f(|{\rm D}^i_t|) + f(|{\rm D}^j_t|) \big).
\end{equation}
The last inequality follows from the trigonometric estimate
\[
    f(|x-y|) \le f(|x|) + f(|y|), \quad x,y \in \T^d,
\]
whose proof is given at the end of this subsection (Lemma \ref{lem:sin-distance}).
Thus,
\[
\mathbb{E} |\Om^i_j| \leq \eta \mathbb{E}f (|{\rm D}^i_t|) + \frac{\eta}{N} \sum_{j=1}^{N}  \mathbb{E}f (|{\rm D}^j_t|).
\]

In order to control $\Upsilon^i_t$, we remark that given $\bar X^i_t$, each random variable
$\bar X^j_t$ and $Y_t^j$ are independent of $\bar X_t^i$ if $i\neq j$, and the law of $(\bar X^j_t, Y_t^j)$ is  identically $\bar \rho_t$. In particular, 
\begin{equation}\label{conditional_ex}
    \mathbb{E} \left[ F (\bar X^i_t - \bar X^j_t, Y^j_t) \,|\, \bar X^i_t \right] = \iint_\Pi F (\bar X^i_t -x, y) \, \d \bar \rho_t(z),
\end{equation}
hence $\mathbb{E} (\Upsilon^i_t) = 0$.
This key observation dates back to Sznitman \cite{Szn91} for the local-in-time propagation of chaos in Vlasov-McKean systems. It is important to note that the inclusion of the type process $Y^j_{t}$ does not interfere with this mechanism; the independence of the limit copies ensures that the conditional expectation property \eqref{conditional_ex} holds, 
thereby preserving the cancellation structure.

By $F(0, \cdot) =0$ and the Cauchy--Schwarz inequality, we have
\begin{align*}
\Upsilon^i_t
=   \frac{1}{N-1}\sum_{j=1, j\neq i}^N F( \bar X^i_t - \bar X^j_t, Y_t^j)  - (F\oast \brho_t)(\bar X^i_t) 
\end{align*}
and 
\begin{align*}
\mathbb E |\Upsilon^i_t|
 \le  & \left( \mathbb E |\Upsilon^i_t|^2\right)^{\frac 12}
=  \left(\mathbb E\lt[ \mathbb E \lt[|\Upsilon^i_t|^2 \,|\, \bar X^i_t \rt] \rt] \right)^{\frac 12}.
\end{align*}
For fixed $j \neq k$, both different from $i$, we have
\begin{align*}
 \mathbb E \lt[ F( \bar X^i_t - \bar X^j_t, Y_t^j)F( \bar X^i_t - \bar X^k_t, Y_t^k) \,|\, \bar X^i_t \rt]   
& = \iiiint_{\Pi \times \Pi} F( \bar X^i_t - x', y')F( \bar X^i_t -x'', y'') \,\d\brt(z') \d\brt(z'')\\
& =(F\oast \brho_t)^2(\bar X^i_t)
\end{align*}
by same reasoning as \eqref{conditional_ex}.
Thus we have
\begin{align*}
\mathbb E \lt[|\Upsilon^i_t|^2 \,|\, \bar X^i_t \rt]&=   \frac{1}{(N-1)^2} \mathbb{E} \lt[ \bigg( \sum_{j\neq i} F( \bar X^i_t - \bar X^j_t, Y_t^j) \bigg)^2 \,\bigg|\,\bar X^i_t\rt]   +  \mathbb{E}\lt[ (F\oast  \bar \rho_t)^2(\bar X^i_t) \,|\,\bar X^i_t\rt] \\
&\quad  -\frac{2}{N-1} \mathbb{E} \lt[ (F\oast  \bar \rho_t)(\bar X^i_t) \sum_{j\neq i} F( \bar X^i_t - \bar X^j_t, Y_t^j) \,|\, \bar X^i_t \rt] \\
& =  \frac{1}{(N-1)^2}\mathbb{E} \lt[\sum_{j\neq i} F^2( \bar X^i_t - \bar X^j_t, Y_t^j) \,|\, \bar X^i_t\rt] \\
&\quad +\frac{2}{(N-1)^2} \sum_{j\neq i,  k\neq i, j<k}\mathbb{E}\lt[ F( \bar X^i_t - \bar X^j_t, Y_t^j)F( \bar X^i_t - \bar X^k_t, Y_t^k) \,|\, \bar X^i_t\rt] \\
&\quad -\mathbb{E} \lt[ (F\oast  \bar \rho_t)^2(\bar X^i_t) \,|\,\bar X^i_t\rt].
\end{align*}
Noting $ |\{1\le j <k \le N : j\neq i, k\neq i\}| = (N-1)(N-2)/2$, 
\begin{align*}
\mathbb E\lt[|\Upsilon^i_t|^2 \,|\, \bar X^i_t \rt]&= \frac{1}{N-1} F^2 \oast \bar \rho_t ( \bar X^i_t) - \frac{1}{N-1} (F\oast \bar \rho_t)^2 ( \bar X^i_t).
\end{align*}
We have
\begin{align*}
   \mathbb{E} |\Upsilon^i_t|^2 & \leq  \frac{1}{N-1}\iiiint_{\Pi \times \Pi}F^2(x-x', y')\, \d\brt(z') \d \brt(z)   \le \frac{\eta^2}{N-1}.  
\end{align*}
Thus, there is a constant $C$, independent of $N$, such that for $N \geq 2$,
\begin{equation*}
    \sup_{t \geq 0} \mathbb{E} |\Upsilon^i_t| \leq C \frac{\eta}{\sqrt N}, \quad i = 1, \dots, N.
\end{equation*}
Hence, by using $f'(x) \leq \frac{\pi}{\sqrt{d}}$ and $2 \psi'_a(x) x \leq 1$, we get
\[
\mathbb{E}{\rm I}^i_t \leq  \frac{\pi}{\sqrt{d}} \eta \mathbb{E}f (|{\rm D}^i_t|) +  \frac{\pi}{\sqrt{d}}\eta\frac{1}{N} \sum_{j=1}^{N}  \mathbb{E}f (|{\rm D}^j_t|) + C   \frac{\eta}{\sqrt N}.
\]

\noindent $\bullet$ Estimate of ${\rm II}^i_t$: Since $\left( \int_0^t {\rm II}^i_s \right)_{t\geq0}$ is a martingale, we readily find
\[
\mathbb{E}{\rm II}^i_t = 0.
\]

\noindent $\bullet$ Estimate of $\rm{ III}^i_t$: Since $\sigma(\cdot)$ is Lipschitz, by using $\psi'_a(x) \leq \frac1{2\sqrt a}$, we get
\[
\mathbb{E} {\rm III}^i_t \leq \frac{(d-1)\pi {\rm Lip}(\sigma)^2}{2\sqrt{d} \sqrt a} \mathbb{E} |Y^i_t - \bar Y^i_t|^2.
\]
Due to the uniform jump rates $\al_1, \al_{-1},$ we deduced \eqref{spin_law_Y}, from which we have
\bq\label{eq_yy}
\mathbb E |Y^{i}_t - \bar Y^{i}_t|^2   
= 2 \mathbb E |Y^{i}_t - \bar Y^{i}_t|   = 2 e^{-(\al_1+\al_{-1})t} \mathbb{E} |Y^i_0 - \bar Y^i_0| = e^{-(\al_1+\al_{-1})t} \mathbb{E} |Y^i_0 - \bar Y^i_0|^2 
\eq
for $t \geq 0$. Thus,
\[
\mathbb{E} {\rm III}^i_t \leq  \frac{(d-1)\pi {\rm Lip}(\sigma)^2}{2\sqrt{d} \sqrt a}  e^{-(\al_1+\al_{-1})t}\mathbb{E} |Y^i_0 - \bar Y^i_0|^2.
\]

\noindent $\bullet$ Estimate of $\rm{ IV}^i_t$: We use
\[
-2 a\psi''_a(r)  = \frac{a}{2(r+a)^{3/2}} \leq \min \left\{\frac{a^{1/4}}{2^{5/2} r^{3/4}}, \  \frac1{2a^{1/2}} \right\}
\]
and
\begin{align*}
& (\sigma(  Y^i_t) -  \sigma( \bar Y^i_t))^2 + 4(\phi_r^\del)^2  ({\rm D}_t^i)\sigma(  Y^i_t) \sigma( \bar Y^i_t) \leq {\rm Lip}(\sigma)^2 |Y^i_t - \bar Y^i_t|^2 + 4\sigma_{\rm max}^2 (\phi_r^\del)^2({\rm D}_t^i)
\end{align*}
to estimate
\[
\begin{aligned}
\mathbb{E} \rm{ IV}^i_t &\leq \frac{\pi {\rm Lip}(\sigma)^2}{2\sqrt{d} \sqrt a} \mathbb{E} |Y^i_t - \bar Y^i_t|^2 +  \frac1{\sqrt 2 \sqrt{d}}\sigma_{\rm max}^2\pi a^{1/4} \mathbb{E}\frac{(\phi_r^\del)^2(| {\rm D}_t^i|)}{ | {\rm D}_t^i|^{3/2}}\cr
&\leq  \frac{\pi {\rm Lip}(\sigma)^2}{2\sqrt{d} \sqrt a} e^{-(\al_1+\al_{-1})t} \mathbb{E} |Y^i_0 - \bar Y^i_0|^2 + 2 \sigma_{\rm max}^2\frac{\pi}{\sqrt{d}} a^{1/4} \frac1{\delta^{3/2}}.
\end{aligned}
\]

\noindent $\bullet$ Estimate of $\rm{ V}^i_t$:  We observe that
\[
(\psi'_a(r))^2  r = \frac{r}{4(r+a)} = \frac14 - \frac{a}{4(r+a)}\le \frac 14.\]
Then we get
\[
2f''(\psi_a( | {\rm D}_t^i|^2))(\psi'_a(| {\rm D}_t^i|^2))^2| {\rm D}_t^i|^2 (\sigma(  Y^i_t) -  \sigma( \bar Y^i_t) )^2 \leq \frac{\pi^2}{2d}{\rm Lip}(\sigma)^2  |Y^i_t - \bar Y^i_t|^2
\]
and
\[
\begin{aligned}
& f''(\psi_a( | {\rm D}_t^i|^2))(\psi'_a(| {\rm D}_t^i|^2))^2| {\rm D}_t^i|^2  (\phi_r^\del)^2  ({\rm D}_t^i)\cr
&\quad = - \frac{\pi^2}{d}f(\psi_a( | {\rm D}_t^i|^2))\left( \frac14 - \frac{a}{4(|{\rm D}_t^i|^2+a)} \right)(\phi_r^\del)^2  ({\rm D}_t^i)\cr
&\quad = - \frac{\pi^2}{4d} f(\psi_a( | {\rm D}_t^i|^2)) +  \frac{\pi^2}{4d} f(\psi_a( | {\rm D}_t^i|^2))(\phi_s^\del)^2  ({\rm D}_t^i) + \frac{a\pi^2 f(\psi_a( | {\rm D}_t^i|^2))}{4d(|{\rm D}_t^i|^2+a)}(\phi_r^\del)^2  ({\rm D}_t^i) \cr
&\quad \leq - \frac{\pi^2}{4d} f(\psi_a( | {\rm D}_t^i|^2)) + \frac{\pi^2}{4d} f(\psi_a( \delta^2))   +\frac{a \pi^3}{2d\sqrt{d} \sqrt{\del^2+4a}}.
\end{aligned}
\]
This together with \eqref{eq_yy} enables us to estimate 
\begin{align*}
\mathbb{E} \rm{ V}^i_t &\leq \frac{\pi^2 {\rm Lip}(\sigma)^2}{2d}e^{-(\al_1+\al_{-1})t}  \mathbb{E} |Y^i_0 - \bar Y^i_0|^2 -2 \sigma_{\rm min}^2 \frac{\pi^2}{d} \mathbb{E} f ( \psi_a( | {\rm D}_t^i|^2))  \cr
&\quad + \frac{2\pi^2}{d} \sigma_{\rm max}^2 f(\psi_a( \delta^2))   +\frac{4a \pi^3  \sigma_{\rm max}^2}{d\sqrt{d} \sqrt{\del^2+4a}}.
\end{align*}

We now collect all the estimates to have
\begin{align*}
 \frac{\d}{\d t} \frac{1}{N} \sum_{i=1}^{N} \mathbb{E} f( \psi_a(  | {\rm D}_t^i|^2)) &\leq -  \frac{2\sigma_{\rm min}^2\pi^2}{d}  \frac{1}{N} \sum_{i=1}^{N}  \mathbb{E} f ( \psi_a( | {\rm D}_t^i|^2)) + \frac{2\pi \eta}{\sqrt{d}}  \frac{1}{N} \sum_{i=1}^{N}  \mathbb{E} f ( | {\rm D}_t^i|) \cr
&\quad + C\frac{\eta}{\sqrt{N}}+ 2 \sigma_{\rm max}^2 \frac{\pi}{\sqrt{d}}  \frac{a^{1/4}}{\delta^{3/2}} + \frac{2\pi^2}{d}  \sigma_{\rm max}^2f(\psi_a( \delta^2)) + 4\sigma_{\rm max}^2 \frac{\pi^3}{d\sqrt{d}}\frac{a}{\sqrt{\del^2+4a}}\cr
&\quad  + \frac{\pi {\rm Lip}(\sigma)^2}{2\sqrt{d}\sqrt{a}}\left( \sqrt{a}\frac{\pi}{\sqrt{d}} + d\right) e^{-(\al_1+\al_{-1})t} \frac1N \sum_{i=1}^N\mathbb{E} |Y^i_0 - \bar Y^i_0|^2.
\end{align*}
We next choose
\[
a \sim \frac{1}{N^5}  \quad \mbox{and} \quad \delta \sim \frac1{\sqrt{N}}, 
\]
then we deduce
\begin{align*}
 \frac{\d}{\d t} \frac{1}{N} \sum_{i=1}^{N} \mathbb{E} f( \psi_a(  | {\rm D}_t^i|^2)) &\leq -  \frac{2\sigma_{\rm min}^2\pi^2}{d}  \frac{1}{N} \sum_{i=1}^{N}  \mathbb{E} f ( \psi_a( | {\rm D}_t^i|^2)) + \frac{2\pi \eta}{\sqrt{d}}  \frac{1}{N} \sum_{i=1}^{N}  \mathbb{E} f ( | {\rm D}_t^i|) \cr
&\quad + \frac{C}{\sqrt{N}} + CN^{\frac52}e^{-(\al_1+\al_{-1})t} \lt(\frac1N \sum_{i=1}^N\mathbb{E} |Y^i_0 - \bar Y^i_0|^2 \rt).
\end{align*}
for $N \geq 1$ large enough, where $C>0$ is independent of $t$ and $N$.
Note that for $r \in [0,\frac{\sqrt{d}}2]$, we have the estimate
\[
|f(r) - f(\psi_a(r^2))| \leq \frac{\pi}{\sqrt d}(\psi_a(r^2) - r)  = \frac{\pi}{\sqrt d}(\sqrt{r^2+ a} - r) \leq \frac{\pi}{\sqrt d}\sqrt{a}.
\]
Combining this with the elementary bound $x \le \frac{\pi}{2}\sin x$ for $x \in [0, \pi/2]$ applied to $f(r) = \sin(\frac{\pi}{\sqrt d} r)$, we obtain
\begin{equation}\label{lin_est}
\frac{\pi}{\sqrt d} r \leq \frac\pi2 f(r) \leq \frac\pi2 \left(f(\psi_a(r^2)) + \frac{\pi}{\sqrt d}\sqrt{a}\right).
\end{equation}
Applying the second inequality in \eqref{lin_est} immediately gives
\[
\frac{1}{N} \sum_{i=1}^{N} \mathbb{E} f( | {\rm D}_t^i|) \leq  \frac{1}{N} \sum_{i=1}^{N} \mathbb{E} f( \psi_a( | {\rm D}_t^i|^2)) + \frac{\pi}{\sqrt d}\sqrt{a} \leq \frac{1}{N} \sum_{i=1}^{N} \mathbb{E} f( \psi_a( | {\rm D}_t^i|^2)) + \frac{C}{N^{\frac 52}}.
\]
 
 Thus, by Gr\"onwall's lemma, we have for $\eta < \sigma_{\rm min}^2 \frac{\pi}{\sqrt{d}}$
\begin{align*}
\frac{1}{N} \sum_{i=1}^{N} \mathbb{E} f( \psi_a( | {\rm D}_t^i|^2))  &\leq \frac{1}{N} \sum_{i=1}^{N} \mathbb{E} f( \psi_a( | {\rm D}_0^i|^2)) e^{-\frac{2\pi}{d} ( \sigma_{\rm min}^2 \pi - \eta \sqrt{d}) t} + \frac{C}{\sqrt{N}} \cr
 &\quad  + CN^{\frac52}e^{-\min\{\frac{2\pi}{d} ( \sigma_{\rm min}^2 \pi - \eta \sqrt{d}), \,(\al_1+\al_{-1})\}t}  \lt(\frac1N \sum_{i=1}^N\mathbb{E} |Y^i_0 - \bar Y^i_0|^2 \rt).
\end{align*}

Hence, we finally arrive at
\begin{align*}
 \frac{1}{N} \sum_{i=1}^{N} \mathbb{E}   | {\rm D}_t^i| &\leq  \frac{C}{N} \sum_{i=1}^{N} \mathbb{E}  | {\rm D}_0^i| e^{-\frac{2\pi}{d} ( \sigma_{\rm min}^2 \pi - \eta \sqrt{d}) t} +   \frac C{\sqrt{N}}  \cr
 &\quad + CN^{\frac52}e^{-\min\{\frac{2\pi}{d} ( \sigma_{\rm min}^2 \pi - \eta \sqrt{d}), \,(\al_1+\al_{-1})\}t} \lt(\frac1N \sum_{i=1}^N\mathbb{E} |Y^i_0 - \bar Y^i_0|^2 \rt)
\end{align*}
for some $C>0$ independent of $t$ and $N$. This completes the proof of Theorem \ref{thm_propa}.
 
 For completeness, we prove here the trigonometric subadditivity estimate used  in  \eqref{Omega_trigonometric}.

\begin{lemma} \label{lem:sin-distance}
Let $f(r) := \sin \big( \frac{\pi}{\sqrt d} r \big)$ for $r \in [0,\tfrac{\sqrt d}{2}]$.
Then for all $x,y \in \T^d$,
\[
  f(|x-y|) \le f(|x|) + f(|y|).
\]
\end{lemma}

\begin{proof}
For $x, y$ in $\T^{d}$, set
\[
 \rho :=  \frac{\pi}{\sqrt d} |x-y| ,  \quad
 r_{1} :=   \frac{\pi}{\sqrt d} |x|, \quad
 r_{2} :=  \frac{\pi}{\sqrt d} |y|,
\]
and assume without loss of generality that $r_{1} \ge r_{2}$. 
Note that $0 \le \rho, r_{1}, r_{2} \le \frac{\pi}{2}$ and
\[
 f(|x-y|) = \sin \rho, \quad
 f(|x|) = \sin r_{1}, \quad
 f(|y|) = \sin r_{2}.
\]

We distinguish two cases according to $r_{1} + r_{2}$.

\medskip
\noindent\textbf{Case 1.} 
$\displaystyle r_{1} + r_{2} \le \frac{\pi}{2}$.

In this case, we have
\[
  \rho \le r_{1} + r_{2} \le \frac{\pi}{2},
\]
and since $\sin$ is increasing on $\big[0,\frac{\pi}{2}\big]$,
\[
  \sin \rho \le \sin(r_{1} + r_{2})
             \le \sin r_{1} + \sin r_{2},
\]
using $\sin(a+b) \le \sin a + \sin b$ for $a,b \in [0,\tfrac{\pi}{2}]$.
Thus $f(|x-y|) \le f(|x|) + f(|y|)$ in this case.

\medskip
\noindent\textbf{Case 2.} 
$\displaystyle \frac{\pi}{2} \le r_{1} + r_{2} \le \pi$.

Since $\sin \rho \le 1$, it suffices to show
\[
  1 \le \sin r_{1} + \sin r_{2}.
\]
Fix $s \in [\frac{\pi}{2},\pi]$ such that $r_{1} + r_{2} = s$, and define
\[
  F(r_{1}) := \sin r_{1} + \sin(s - r_{1}).
\]
A direct  computation shows that $F$ has a unique critical point
$r_{1} = \frac{s}{2}$ on $(0,\frac{\pi}{2})$ (when $s \neq \pi$), and
$F''(\frac{s}{2}) < 0$, so this point is a local maximum. 
Hence the minimum of $F$ on the interval $\big[0,\frac{\pi}{2}\big]$ is attained at the endpoint $r_{1} = \frac{\pi}{2}$. Since $\frac{\pi}{2} \le s \le \pi$, we obtain
\[
  F \left(\frac{\pi}{2}\right)
  = 1 + \sin \left(s - \frac{\pi}{2}\right) \ge 1.
\]
Thus $1 \le \sin r_{1} + \sin r_{2}$ as desired.
\end{proof}
 
%
%
%
%
%
%

  \section{Long-time exponential contraction}\label{sec_lt}
  
  In this section, we establish a uniform exponential contraction estimate between two solutions of the McKean--Vlasov equation \eqref{new_pde}.
As a consequence, we obtain the exponential convergence of any solution toward the unique stationary equilibrium distribution.

The proof relies on a stability argument closely related to the coupling analysis developed in Section \ref{sec_poc}.
More precisely, the long-time contraction result can be viewed as the infinite-particle $(N\to\infty)$ analogue of the uniform-in-time propagation of chaos estimate: 
the empirical measure $\mu_t^N$ is replaced by another McKean--Vlasov solution $\bar\gamma_t$ with a different initial law.
In this mean-field setting, the finite-$N$ statistical fluctuations disappear, and the resulting estimate reduces to a pure exponential contraction in the Wasserstein distance.

Since the arguments largely parallel those used in the proof of Theorem \ref{thm_propa}, we only highlight the necessary modifications and focus on the terms that are specific to the mean-field contraction.
 
\begin{theorem}\label{thm_lt} Assume that \eqref{hyp_F} holds and $\eta>0$ sufficiently small so that $\eta < \sigma_{\rm min}^2 \frac{4 \pi }{3d}$. 
Let $\bar\rho_t = {\rm Law}(\xt,  \yt)$ and $\bar\gamma_t = {\rm Law}(P_t, Q_t)$ be two McKean--Vlasov solutions of \eqref{new_pde} subject to the initial data $\rho_0,\gamma_0\in \calP(\T^d\times\{\pm1\})$. Then there exists a positive constant $C$, independent of $t$, such that 
\[
\calW_1( \bar \rho_t, \bar \ga_t) \leq C\lt(\calW_1(  \rho_0,   \ga_0) + \calW_1(  \rho_0,   \ga_0)^\frac17\rt) e^{ - \frac17\min\{ c_*, \alpha_1 + \alpha_{-1} \} t},
\]
 where 
 \[
 c_* := \frac{2\sigma_{\rm min}^2\pi^2}{d} -  \frac{3\pi}{2}\eta > 0.  
 \]
 In particular, if $\rho_\infty$ is a stationary distribution of \eqref{new_pde}, then we obtain the exponential convergence
 \[
\calW_1( \bar \rho_t, \rho_\infty) \leq C\lt(\calW_1(  \rho_0,  \rho_\infty) + \calW_1(  \rho_0,  \rho_\infty)^\frac17\rt) e^{ - \frac17\min\{ c_*, \alpha_1 + \alpha_{-1} \} t}.
\]
\end{theorem}   
 
%
%
%
%
%
%
\subsection{Proof of Theorem \ref{thm_lt}}
  
  We consider two McKean--Vlasov solutions $(X_t,Y_t)$ and $(P_t,Q_t)$ of \eqref{new_pde} with initial laws $\rho_0$ and $\gamma_0$, respectively.
Our goal is to estimate the evolution of the Wasserstein distance between their laws by means of a suitable coupling.

As in Section \ref{sec_propa}, we couple the spatial components using the reflection coupling introduced in Section \ref{reflection coupling}, while the internal states evolve independently according to their jump dynamics. The position processes $X_t$ and $P_t$ are thus coupled as follows:
\begin{align*}
\begin{aligned}
\d X_t &=  \phi_r^\delta(E_t)(I_d - 2 e_t e_t^T){\sigma(Y_t)}  \, \d B_t + \phi_s^\delta(E_t) {\sigma(Y_t)}  \, \d\tilde B_t + F \oast  \bar\rho_t(X_t)\,  \d t,\\
\d P_t &=  \phi_r^\delta(E_t) {\sigma(Q_t)} \, \d B_t + \phi_s^\delta(E_t) {\sigma(Q_t)} \, \d\tilde B_t +  F \oast  \bar\gamma_t(P_t)\, \d t, \\ 
\end{aligned}
\end{align*}
where  $\rho_t$ and $\eta_t$ denote the law of the processes $(X_t, Y_t)$ and $(P_t, Q_t)$, respectively, with $z = (x, y)$ and $w = (p, q)$.

We define the spatial discrepancy between the two processes by
\[
{\rm E}_t := X_t - P_t.
\] 
Then by using almost the same arguments as in the proof of Theorem \ref{thm_propa}, we obtain
\begin{align*}
 \frac{\d}{\d t} \mathbb E f( \psi_a( | {\rm E}_t|^2))  
&= 2 \mathbb E f'(\psi_a( | {\rm E}_t|^2)) \psi'_a(| {\rm E}_t|^2) \langle  {\rm E}_t, (F \oast \bar\rho_t (X_t) -  F \oast \bar\gamma_t (P_t) )\rangle   \cr
&\quad + 2 \mathbb E f'(\psi_a( | {\rm E}_t|^2)) \psi'_a(| {\rm E}_t|^2) \langle {\rm E}_t,A_t \, \d B_t + \tilde A_t \, \d\tilde B_t \rangle \cr
&\quad + (d-1) \mathbb E f'(\psi_a( | {\rm E}_t|^2)) \psi'_a(| {\rm E}_t|^2) (\sigma(  Y_t) -  \sigma( Q_t) )^2  \cr
&\quad   -  2a\mathbb E f'(\psi_a( | {\rm E}_t|^2)) \psi''_a(| {\rm E}_t|^2) \big\{  (\sigma(  Y_t) -  \sigma( Q_t))^2 + 4  (\phi_r^\del)^2  ({\rm E}_t)\sigma(  Y_t) \sigma( Q_t)\big\}  \cr
&\quad + 2\mathbb E f''(\psi_a( | {\rm E}_t|^2))(\psi'_a(| {\rm E}_t|^2))^2 \big\{ |(A_t)^T {\rm E}_t|^2   + |(\tilde A_t)^T {\rm E}_t|^2  \big\},
\end{align*}
where
\begin{align*}
A_t &:= \phi_r^\del ({\rm E}_t)(I_d -2 e_t \otimes e_t )  {\sigma(  Y_t)} -      \phi_r^\del ({\rm E}_t) { \sigma( Q_t) } =  \phi_r^\del ({\rm E}_t) \{ {\sigma(  Y_t)} - { \sigma( Q_t) }  \} I_d - 2   \phi_r^\del ({\rm E}_t) e_t \otimes e_t  {\sigma(  Q_t)}
\end{align*}
and
\[
\tilde A_t := \phi_s^\del ({\rm E}_t)  \{ {\sigma(  Y_t)} - { \sigma( Q_t) }  \} I_d.
\]
For the first term on the right-hand side, we observe
 \begin{align*}
& \int_{\T^d} F(X_t-x,y) \bar \rho_t(x, 1) - F (P_t- x, 1) \ga_t(x,1)\, \d x  \\
&\quad \le   \int_{\T^d} (F(X_t - x,1) - F(P_t -x,1)) \bar \rho(x, 1)\, \d x + 
 \int_{\T^d}   F(P_t-x,1) (\bar \rho_t(x,1) - \bar \ga_t(x,1))\, \d x \\
&\quad  \le  {\rm Lip}(F)  |{\rm E}_t|  +  {\rm Lip}(F) \calW_1(\bar \rho_t(\cdot,1), \bar \ga_t(\cdot, 1)),
\end{align*}
where we use the assumption $(\bf{H2})$ and the Kantorovich--Rubinstein duality for $\calW_1$ norm. 
Treating similarly the other half and adding the two,  we have
\[
 \left|  \iint_{\Pi} F(X_t-x,y)\, \d \bar \rho_t(z) - \iint_{\Pi} F (P_t- x, -1) \, \d \bar\ga_t(w) \right|   \le 2  {\rm Lip}(F)  |{\rm E}_t|  +  {\rm Lip}(F)  \calW_1(\bar \rho_t, \bar \ga_t).
\]
This implies
\[
2 \mathbb E f'(\psi_a( | {\rm E}_t|^2)) \psi'_a(| {\rm E}_t|^2) \langle  {\rm E}_t, (F \oast \bar\rho_t (X_t) -  F \oast \bar\gamma_t (P_t) )\rangle \leq 3\eta\calW_1(\bar \rho_t, \bar \ga_t).
\]
The estimates of the rest are parallel with those in Section \ref{sec_propa}. Hence,
 \begin{align*}
 \frac{\d}{\d t} \mathbb E f( \psi_a( | {\rm E}_t|^2))  &\leq 3\eta\calW_1(\bar \rho_t, \bar \ga_t)- \frac{2\sigma_{\rm min}^2\pi^2}{d}  \mathbb E f( \psi_a( | {\rm E}_t|^2))   + 2 \sigma_{\rm max}^2 \frac{\pi}{\sqrt{d}}  \frac{a^{1/4}}{\delta^{3/2}}\cr
&\quad + \frac{2\pi^2}{d}  \sigma_{\rm max}^2f(\psi_a( \delta^2)) + 4\sigma_{\rm max}^2 \frac{\pi^3}{d\sqrt{d}}\frac{a}{\sqrt{\del^2+4a}}\cr
&\quad  + \frac{\pi {\rm Lip}(\sigma)^2}{2\sqrt{d}\sqrt{a}}\left( \sqrt{a}\frac{\pi}{\sqrt{d}} + d\right)e^{-(\al_1+\al_{-1})t}  \mathbb{E} |Y_0 - Q_0|.
 \end{align*}
 We next choose $\del \sim a^{\frac 19 }$ with $a \leq 1$ to deduce
\[
\frac{\d}{\d t}   \mathbb{E} f( \psi_a( | {\rm E}_t|^2)) \leq - \frac{2\sigma_{\rm min}^2\pi^2}{d}  \mathbb{E} f ( \psi_a( | {\rm E}_t|^2)) + 3\eta \calW_1(\bar \rho_t, \bar \ga_t) + 
\frac{C}{\sqrt{a}}e^{-(\al_1+\al_{-1})t} \mathbb{E} |Y_0 - Q_0| + C a^{\frac{1}{12}}.
\]
Here, we use \eqref{lin_est} to get
 \[
 \calW_1(\bar \rho_t, \bar \ga_t) \leq \mathbb{E} |{\rm E}_t| + \mathbb{E} |Y_t - Q_t| \leq \frac\pi2 \lt(\mathbb{E}f( \psi_a( | {\rm E}_t|)) + \frac{\pi}{\sqrt{d}}\sqrt{a}\rt) + e^{-(\al_1+\al_{-1})t} \mathbb{E} |Y_0 - Q_0|,
 \]
 and subsequently, we obtain
 \[
 \frac{\d}{\d t}  \mathbb{E} f( \psi_a( | {\rm E}_t|^2))   \leq - \left( \frac{2\sigma_{\rm min}^2\pi^2}{d} -  \frac{3\pi}{2}\eta\right)   \mathbb E f(\psi_a(|{\rm E}_t|^2))   +\frac{C}{\sqrt{a}}e^{-(\al_1+\al_{-1})t} \mathbb{E} |Y_0 - Q_0| + C a^{\frac{1}{12}}.
 \]
 Applying Gr\"onwall's lemma gives
 \begin{align*}
  \mathbb{E} f( \psi_a( | {\rm E}_t|^2))  & \leq    \mathbb{E} f( \psi_a( | {\rm E}_0|^2))  e^{-c_* t} + \frac{C \mathbb{E} |Y_0 - Q_0|^2}{\sqrt{a} (\alpha_1 + \alpha_{-1} - c_*)}\left(e^{- c_* t} - e^{- (\alpha_1 + \alpha_{-1})t} \right)  + Ca^{\frac1{12}}(1 - e^{-c_* t})\cr
 &\leq    \mathbb{E} f( \psi_a( | {\rm E}_0|^2))   e^{-c_* t} + \frac{C \mathbb{E} |Y_0 - Q_0|^2}{\sqrt{a} |\alpha_1 + \alpha_{-1} - c_*|}e^{- \min\{c_*, (\alpha_1 + \alpha_{-1})\} t}  + Ca^{\frac1{12}},
 \end{align*}
 where
 \[
 c_* =  \frac{2\sigma_{\rm min}^2\pi^2}{d} -  \frac{3\sqrt d}{2}\eta 
 \]
 and we chosen $\eta > 0$ small enough so that $c_* > 0$ and $\alpha_1 + \alpha_{-1} \neq c_*$.  Since the above inequality holds for any $a > 0$, we can choose 
 \[
 a =  \lt( \frac{\mathbb{E} |Y_0 - Q_0|^2}{\sqrt{a} |\alpha_1 + \alpha_{-1} - c_*|}e^{- \min\{c_*, (\alpha_1 + \alpha_{-1})\} t}\rt)^\frac{12}{7},
 \]
 so that
 \[
 \frac{\mathbb{E} |Y_0 - Q_0|^2}{\sqrt{a} |\alpha_1 + \alpha_{-1} - c_*|}e^{- \min\{c_*, (\alpha_1 + \alpha_{-1})\} t} = a^{\frac1{12}}.
 \]
This deduces
 \begin{align*}
\mathbb{E} |{\rm E}_t|  &\leq   \frac{\sqrt d}{2}      \mathbb{E} f( \psi_a( | {\rm E}_0|^2)) e^{-c_* t}  +  C \frac{ (\mathbb{E} |Y_0 - Q_0|)^\frac17}{  |\alpha_1 + \alpha_{-1} - c_*|^\frac17}e^{- \frac17 \min\{c_*, (\alpha_1 + \alpha_{-1})\} t}
 \end{align*}
 due to $\mathbb{E} |Y_0 - Q_0|^2 = 2\mathbb{E} |Y_0 - Q_0|$. This, combined with the decay estimate of $\mathbb{E} |Y_t - Q_t|$, completes the proof.
 
%
%
%
%
%
%
%
%
%

 \section{Bifurcation for large $\eta$}\label{sec_bif}

Theorem \ref{thm_lt} asserts the uniqueness of a static (stationary) solution to \eqref{pde} when the nonlocal interaction has a sufficiently small Lipschitz constant $\eta$. We now turn to the complementary regime of stronger interactions. Our goal is to show that, once $\eta$ exceeds a computable threshold $\eta_*$, the spatially homogeneous equilibrium may lose stability and give rise to nontrivial stationary branches through a bifurcation mechanism.

Our strategy is inspired by the bifurcation analysis developed in \cite{CGPS} for single-species McKean--Vlasov equations on $\T^d$. However, the present setting involves a genuinely coupled two-species system, with interconversion rates $\alpha_1,\alpha_{-1}$ and distinct diffusion coefficients $\sigma^2(1),\sigma^2(-1)$. This additional structure leads to a $2\times2$ mode interaction at the linearized level, and the critical threshold is determined by the spectrum of a matrix-valued Fourier symbol. Accordingly, the bifurcation analysis requires a nontrivial extension of the single-species framework to accommodate the coupling effects.

Throughout this section, we assume that the interaction potentials $U$ and $V$ are
coordinate-wise even and restrict ourselves to coordinate-wise even perturbations.

\subsection{Fourier basis on the even subspace}

We define
\[
L_s^2(\T^d)
:=
\Big\{
f\in L^2(\T^d):
f(x_1,\dots,x_i,\dots,x_d)=f(x_1,\dots,-x_i,\dots,x_d)
\ \text{for all } i=1,\dots,d
\Big\},
\]
and its zero-mean subspace
\[
L_{0,s}^2(\T^d)
:=
\Big\{
f\in L_s^2(\T^d):
\int_{\T^d} f(x)\,\dx = 0
\Big\}.
\]
We then set
\[
X_s := L_{0,s}^2(\T^d)\times L_{0,s}^2(\T^d).
\]

Following \cite{CGPS}, we introduce a real orthonormal basis $\{w_k\}_{k\in\Z^d}$ of $L^2(\T^d)$:
\[
w_k(x)
=
N_k \prod_{i=1}^d w_{k_i}(x_i),
\quad
k=(k_1,\dots,k_d)\in\Z^d,
\]
where
\[
w_{k_i}(x_i)
=
\begin{cases}
\cos (2\pi k_i x_i), & k_i>0,\\[2mm]
1, & k_i=0,\\[2mm]
\sin (2\pi |k_i| x_i), & k_i<0,
\end{cases}
\quad
N_k
=
\prod_{i=1}^d (2-\delta_{k_i,0})^{1/2}.
\]
For any $f\in L^2(\T^d)$, we define its (real) Fourier coefficients by
\[
\tilde f(k) := \langle f, w_k\rangle_{L^2(\T^d)}.
\]
The subspace $L_s^2(\T^d)$ is closed in $L^2(\T^d)$, and one checks that
$
\{w_k\}_{k\in\N_0^d}$
forms an orthonormal basis of $L_s^2(\T^d)$, where we denote
$\N_0=\{0,1,2,\dots\}$.
In particular, any $f\in L_{0,s}^2(\T^d)$ admits the expansion
\[
f(x)=\sum_{k\in\N_0^d\setminus\{0\}} \tilde f(k)\, w_k(x).
\]

\subsection{Stationary solutions as zeroes of $\Phi$}
Replacing the interaction potential by $\eta F$ with the normalization $\mathrm{Lip}(F)\le 1$,
a stationary solution $(u,v)$ of \eqref{pde} for a given $\eta$ solves
\begin{align}
\label{eq:stat}
\begin{aligned}
0&= \frac{\sigma^2(1)}{2}\Delta u - \eta \nabla \cdot (u B(u,v)) - \alpha_1 u + \alpha_{-1} v,\\
0&= \frac{\sigma^2(-1)}{2}\Delta v - \eta \nabla \cdot (v B(u,v)) + \alpha_1 u - \alpha_{-1} v,
\end{aligned}
\end{align}
with the nonlocal drift
\[
B(u,v) = \nabla U*u + \nabla V*v.
\]
Note that $B(c_1,c_2)\equiv 0$ for any constants $c_1,c_2$.

Balancing the interconversion terms yields the homogeneous equilibrium
\[
u_c = \frac{\alpha_{-1}}{\alpha_1+\alpha_{-1}},
\quad
v_c = \frac{\alpha_1}{\alpha_1+\alpha_{-1}},
\]
so that $(u,v)=(u_c,v_c)$ solves \eqref{eq:stat} for every $\eta\ge 0$.
To analyze nontrivial stationary states branching from $(u_c,v_c)$, we set
\[
m:=u-u_c,\quad n:=v-v_c.
\]
Recalling that $\int_{\T^{d}} (u + v)\,dx =1$, integrating \eqref{eq:stat} gives
$\int_{\T^d} u\,dx=u_c$ and $\int_{\T^d} v\,dx=v_c$, hence
\[
\int_{\T^{d}} m\,\dx=\int_{\T^{d}} n\,\dx=0.
\]
Using $B(u_c,v_c)=0$, we obtain
\begin{align}
\label{eq:pert}
\begin{aligned}
0 &= \frac{\sigma^2(1)}{2} \Delta m
- \eta \nabla \cdot \big(u_c B(m,n) + m B(m,n)\big)
- \alpha_1 m + \alpha_{-1} n,\\
0 &= \frac{\sigma^2(-1)}{2} \Delta n
- \eta \nabla \cdot \big(v_c B(m,n) + n B(m,n)\big)
+ \alpha_1 m - \alpha_{-1} n.
\end{aligned}
\end{align}

It is convenient to invert the elliptic parts and write \eqref{eq:pert} as a functional fixed-point problem.
Let
\[
\beta_\ell^{-1} := \frac{\sigma^2(\ell)}{2}, \quad \ell=\pm1,
\]
and define
\[
\mathcal{L}_\ell := -\beta_\ell^{-1}\Delta + \alpha_\ell
:\ H^2(\T^d)\cap L^2_0(\T^d) \longrightarrow L^2_0(\T^d),
\]
which is invertible on $L^2_0(\T^d)$, the zero-mean subspace of $L^2(\T^d)$, and whose inverse preserves the subspace $L^2_{0,s}(\T^d)$.
On $L^2_{0,s}(\T^d)$ we may write the inverse as convolution with the cosine-series kernel
\[
\Omega_\ell(x)
  = \sum_{k\in\N_0^d\setminus\{0\}}
    \frac{1}{D_\ell(k)}\, w_k(x),
\quad
D_\ell(k) := \frac{|2\pi k|^2}{\beta_\ell} + \alpha_\ell,
\]
so that $\mathcal L_\ell^{-1} f = \Omega_\ell * f$ for all $f\in L^2_{0,s}(\T^d)$.

For $k\in\N_0^d$ and $f\in L^2_{s}(\T^d)$ we have the multiplier identities
\[
(\Delta f)\,\tilde{}(k)=-|2\pi k|^2\,\tilde f(k),\quad
(U*f)\,\tilde{}(k)=\widetilde U(k)\,\tilde f(k),\quad
(\nabla\cdot(\nabla U*f))\,\tilde{}(k)=-|2\pi k|^2\,\widetilde U(k)\,\tilde f(k),
\]
and similarly for $V$, where $\widetilde U(k)$ and $\widetilde V(k)$ denote the cosine coefficients
of $U$ and $V$ on $\T^d$.

Rewriting \eqref{eq:pert}, we obtain the functional system
\begin{align}
\label{eq:functional-eq}
\begin{aligned}
m &= -\eta \Omega_{1}\ast(\nabla \cdot \mathfrak M) + \alpha_{-1} \Omega_1*n,\\
n &= -\eta \Omega_{-1}\ast(\nabla \cdot \mathfrak N) + \alpha_1 \Omega_{-1}*m,
\end{aligned}
\end{align}
where
\[
\mathfrak M := u_c B(m,n) + m B(m,n),
\quad
\mathfrak N := v_c B(m,n) + n B(m,n).
\]
Note that $\int_{\T^d} \Omega_{\ell} * f \,\dx=0$ for any $f\in L^2_0(\T^d)$.
We then introduce the nonlinear map $\Phi: X_s\times\R\to X_s$ by
\begin{align}\label{eq:Phi}
\Phi((m,n),\eta)
:=
\Big(
m + \eta \Omega_1*(\nabla \cdot \mathfrak M) - \alpha_{-1}\Omega_1*n,\ 
n + \eta \Omega_{-1}*(\nabla \cdot \mathfrak N) - \alpha_1\Omega_{-1}*m
\Big).
\end{align}
Then $\Phi((m,n),\eta)=0$ is equivalent to \eqref{eq:functional-eq}, and clearly $\Phi(0,\eta)=0$ for all $\eta$.\\
\indent
To justify further analysis, such as the application of bifurcation theory or fixed-point arguments, we first verify that $\Phi$ is indeed well-defined on the space $X_s$.
\begin{lemma}
Assume $\nabla U, \nabla V \in W^{1,\infty}(\T^d)$.
Then for every $(m,n)\in X_s$ and $\eta\in\R$, the map $\Phi$ in \eqref{eq:Phi} is well-defined and satisfies
$\Phi((m,n),\eta)\in X_s$.
\end{lemma}

\begin{proof}
By monotonicity of $L^p$ on a bounded domain,
\[
\|m\|_{L^1}\le C\|m\|_{L^2},\quad \|n\|_{L^1}\le C\|n\|_{L^2}.
\]
With $\nabla U,\nabla V\in L^\infty(\T^d)$, Young's inequality yields
\[
\|\nabla U*m\|_{L^\infty}\le C\|\nabla U\|_{L^\infty}\|m\|_{L^2},
\quad
\|\nabla V*n\|_{L^\infty}\le C\|\nabla V\|_{L^\infty}\|n\|_{L^2}.
\]
Hence
\[
\|B(m,n)\|_{L^\infty}\le C\big(\|m\|_{L^2}+\|n\|_{L^2}\big).
\]
For the nonlinear products,
\[
\|u_c B(m,n)\|_{L^2}\le C\|B(m,n)\|_{L^\infty},
\quad
\|mB(m,n)\|_{L^2}\le \|m\|_{L^2}\|B(m,n)\|_{L^\infty},
\]
and similarly for $\mathfrak N$. Therefore
\[
\|\mathfrak M\|_{L^2}+\|\mathfrak N\|_{L^2}
\le C\Big(\|m\|_{L^2}+\|n\|_{L^2}+\|m\|_{L^2}^2+\|n\|_{L^2}^2\Big).
\]

\medskip
We will use complex Fourier coefficients for $L^2$ multiplier bounds.
For $f\in L^2(\T^d)$ we write
\[
\widehat f(k):=\int_{\T^d} f(x)\,e^{-2\pi i k\cdot x}\,\dx,\quad k\in\Z^d.
\]
Then
\[
\widehat{\Delta f}(k)=-|2\pi k|^2\widehat f(k),
\quad
\widehat{U*f}(k)=\widehat U(k)\,\widehat f(k),
\]
and similarly for $V$.\\
\indent
 For $\ell=\pm1$,  $\Omega_\ell$ is the Fourier multiplier with symbol $1/D_\ell(k)$, $k\neq 0$, 
\[
D_\ell(k)=\frac{ |2\pi k|^2}{\beta_\ell}+\alpha_\ell \ge \frac{(2\pi)^2}{\beta_\ell},
\]
thus 
\[\widehat{\Omega_\ell*h}(k)=\widehat h(k)/D_\ell(k), \quad
\widehat{\Omega_\ell*\nabla\cdot q}(k)=(2\pi i k\cdot\widehat q(k))/D_\ell(k).\]
By Plancherel's identity, we obtain
\[
\|\Omega_\ell*h\|_{L^2}^2
=\sum_{k\in\Z^d \setminus\{0\}}\frac{|\widehat h(k)|^2}{D_\ell(k)^2}
\le  \frac{\beta_\ell^2}{(2\pi)^2} \sum_{k\in\Z^d\setminus \{0\}}|\widehat h(k)|^2
=\frac{\beta_\ell^2}{(2\pi)^2}  \|h\|_{L^2}^2
\]
for $h\in L^{2}_{0}(\T^{d})$, hence
\[
\|\Omega_\ell*h\|_{L^2}\le \frac{\beta_\ell}{2\pi} \|h\|_{L^2}.
\]

{For the divergence term, define
\[
\phi_\ell(r) = \frac{2\pi r}{(2\pi r)^2/\beta_\ell+\alpha_\ell} 
= \frac{2\pi \beta_\ell r}{(2\pi r)^2+\alpha_\ell\beta_\ell}, 
\quad r=|k| \ge 1.
\]
A simple maximization shows that $\phi_{\ell}$ attains the maximum at $\max\{ 1, \frac{\sqrt{\al_{\ell}\be_{\ell}}}{2\pi}\}$, and
\[
\sup_{r\ge 1}\phi_\ell(r) \le \max\lt\{ \phi_\ell(1), \frac{\sqrt{\beta_\ell}}{2\sqrt{\alpha_\ell}} \rt\}.   
\]
Therefore, similarly as before, by Plancherel's identity, we have
\[
\|\Omega_\ell*\nabla \cdot q\|_{L^2} \leq \frac{\beta_\ell}{2\pi} \|q\|_{L^2}.
\]}

Combining the estimates, we deduce
\[
\|\Omega_\ell*n\|_{L^2}\le C\|n\|_{L^2},
\quad
\|\Omega_\ell*(\nabla\cdot \mathfrak M)\|_{L^2}\le C\|\mathfrak M\|_{L^2},
\quad
\|\Omega_\ell*(\nabla\cdot \mathfrak N)\|_{L^2}\le C\|\mathfrak N\|_{L^2}.
\]
Hence, for all $(m,n)\in X_s$,
\[
\|\Phi((m,n),\eta)\|_{X_s}
\le \|(m,n)\|_{X_s}
+ C(|\eta|+1)\big(\|(m,n)\|_{X_s}+\|(m,n)\|_{X_s}^2\big).
\]

Finally, since $U$ and $V$ are coordinate-wise even, $\nabla U$ and $\nabla V$ are coordinate-wise odd.
Thus for $(m,n)\in X_s$, the drift field $B(m,n)$ is coordinate-wise odd, while
$\nabla\cdot(u_cB(m,n))$, $\nabla\cdot(mB(m,n))$ and similarly the terms in $\nabla\cdot\mathfrak N$
are coordinate-wise even.
Therefore $\Phi(\cdot,\eta)$ maps $X_s$ into itself.
\end{proof}

\medskip
In order to establish the existence of nontrivial solution branches bifurcating from the trivial equilibrium,
we use the classical Crandall--Rabinowitz theorem \cite{CR71}; see also \cite{K12} and \cite[Appendix A]{CGPS}.
For completeness, we recall the statement below.

\begin{theorem}\label{C-R}{\rm (\cite[Appendix A.2]{CGPS})}
Consider a separable Hilbert space $X$ with $U\subset X$ an open neighbourhood of $0$,
and a nonlinear $C^2$ map $F:U\times V \to X$, where $V$ is an open subset of $\R^+$
such that $F(0,\kappa)=0$ for all $\kappa\in V$.
Assume that for some $\kappa_*\in V$:
\begin{enumerate}
  \item $D_xF(0,\kappa_*)$ is Fredholm of index zero and has a one-dimensional kernel;
  \item $D^2_{x\kappa}F(0,\kappa_*)[\hat v_0]\notin \mathrm{Im}\,D_xF(0,\kappa_*)$ for some
  $\hat v_0\in\ker D_xF(0,\kappa_*)$ with $\|\hat v_0\|=1$.
\end{enumerate}
Then $(0,\kappa_*)$ is a bifurcation point, and there exists a nontrivial $C^1$ curve of solutions
$(x(s),\kappa(s))$ with $x(0)=0$, $\kappa(0)=\kappa_*$, and $F(x(s),\kappa(s))=0$ for $|s|$ small.
\end{theorem}

Applying Theorem \ref{C-R} with $X=X_s$ and $\kappa=\eta$ yields the following bifurcation criterion,
which is the analogue of \cite[Theorem 4.2]{CGPS} adapted to the present two-species system.

\begin{theorem}[Bifurcation criterion]\label{bifurcation}
For $k\in\N_0^d\setminus\{0\}$, define
\[
\eta_k
:=
\frac{(\alpha_1+\alpha_{-1})\big(D_1(k)D_{-1}(k)-\alpha_1\alpha_{-1}\big)}
{|2\pi k|^2\big[\alpha_1\,\widetilde V(k)\,(D_1(k)+\alpha_{-1})
+\alpha_{-1}\,\widetilde U(k)\,(D_{-1}(k)+\alpha_1)\big]}.
\]
Assume that the interaction potentials $U$ and $V$
are such that there exists a mode
$k^*\in\N_0^d\setminus\{0\}$
for which the following conditions hold:
\begin{enumerate}
\item[(i)] $\eta_{k^*}$ is finite and
\[
\mathrm{card}\Big\{k\in\N_0^d\setminus\{0\}:\ \eta_k=\eta_{k^*}\Big\}=1;
\]
\item[(ii)] $\eta_{k^*}>0$, equivalently,
\[
\alpha_1\,\widetilde V(k^*)\,(D_1(k^*)+\alpha_{-1})
+\alpha_{-1}\,\widetilde U(k^*)\,(D_{-1}(k^*)+\alpha_1)>0;
\]
\item[(iii)] $D_r\Phi(0,\eta_*)$ is self-adjoint on $X_s$, equivalently,
\[
(\alpha_{-1} D_{-1} - \alpha_1 D_1)\,
(\alpha_1 \widetilde V (D_1+\alpha_{-1}) + \alpha_{-1}\widetilde U (D_{-1}+\alpha_1))
=
(D_1 D_{-1} - \alpha_1\alpha_{-1})\,
(\alpha_1 \widetilde U D_1 - \alpha_{-1}\widetilde V D_{-1}),
\]
where all quantities are evaluated at $k=k^*$.
\end{enumerate}
Then $\Phi$ satisfies the hypotheses of Theorem \ref{C-R} on $X_s$,
and $(0,\eta_{k^*})\in X_s\times\R^+$ is a bifurcation point of $\Phi(r,\eta)=0$.
More precisely, there exists a nontrivial $C^1$ curve of solutions
$\{(r(s),\eta(s))\}_{|s|<\delta}$ for some $\delta>0$ with $r(0)=0$, $\eta(0)=\eta_{k^*}$, and
\[
r'(0)\in \ker D_r\Phi(0,\eta_*)=\mathrm{span}\{(w_{k^*},c\,w_{k^*})\}
\quad\text{for some }c\neq 0,
\]
where $\eta_*:=\eta_{k^*}$.
\end{theorem}

\begin{remark}[On the simplicity assumption]
The restriction to the coordinate-wise even subspace $X_s$ is mainly a technical choice to avoid the $\pm k$
degeneracy of Fourier modes. On the full space $L^2_0(\T^d)$ with the complex basis $\{e^{2\pi i k\cdot x}\}_{k\in\Z^d}$,
the modes $k$ and $-k$ typically produce a two-dimensional kernel, so the Crandall--Rabinowitz theorem
cannot be applied directly without additional symmetry reduction.
\end{remark}

\begin{remark}[Special cases]
We record a few situations in which the formula for the critical parameter
and the conditions in \textup{(ii)} and \textup{(iii)} simplify.

\begin{enumerate}

\item \textup{(Symmetric diffusion and symmetric flip rates).}
Assume $\alpha_1=\alpha_{-1}=\alpha$ and $\sigma(1)=\sigma(-1)=\sigma$.
Then $D_1(k)=D_{-1}(k)\equiv D(k)$ and $u_c=v_c=\tfrac12$.
In this case, the critical parameter simplifies to
\[
\eta_k
=\frac{\sigma^2}{\widetilde U(k)+\widetilde V(k)}.
\]
Thus the simplicity condition \textup{(i)}
is equivalent to requiring that
$k^*$ is the unique minimizer of
$\widetilde U(k)+\widetilde V(k)$ over
$k\in\N_0^d\setminus\{0\}$.

Moreover, the left-hand side of \textup{(iii)} vanishes identically, so
\[
0=(D(k^*)^2-\alpha^2)
\bigl(\widetilde U(k^*)-\widetilde V(k^*)\bigr).
\]
Since $D(k^*)>\alpha$ for $k^*\neq0$, condition \textup{(iii)}
reduces to
\[
\widetilde U(k^*)=\widetilde V(k^*).
\]
Under this symmetry, condition \textup{(ii)} becomes
\[
\widetilde U(k^*)=\widetilde V(k^*)>0.
\]

\item \textup{(One-sided flip rates).}
Assume $\alpha_{-1}=0$.
Then $u_c=0$, $v_c=1$, and
\[
D_{-1}(k)
=\frac{|2\pi k|^2}{\beta_{-1}}
=\frac{\sigma^2(-1)}{2}\,|2\pi k|^2.
\]
Substituting $\alpha_{-1}=0$ into the general formula for $\eta_k$,
we obtain
\[
\eta_k
=\frac{D_{-1}(k)}{|2\pi k|^2\,\widetilde V(k)}
=\frac{\sigma^2(-1)}{2\,\widetilde V(k)}.
\]
In this setting, the simplicity condition \textup{(i)}
reduces to the requirement that
$k^*$ is the unique minimizer of $\widetilde V(k)$.

Condition \textup{(ii)} is  simply
\[
\widetilde V(k^*)>0,
\]
and
Condition \textup{(iii)} reduces to
\[
\frac{\widetilde U(k^*)}{\widetilde V(k^*)}
=-\frac{\alpha_1}{D_{-1}(k^*)}
=-\frac{2\alpha_1}{\sigma^2(-1)\,|2\pi k^*|^2}.
\]

\end{enumerate}
\end{remark}

%
%
%
%
%
%

\subsection{Proof of Theorem \ref{C-R}} 
We verify the two hypotheses of Theorem \ref{C-R} for $F=\Phi$ on $X_s$.

\medskip
\noindent{\it Verification of Condition (1).}
At linear order, $B(m,n)=\nabla U*m+\nabla V*n$, so the Fr\'echet derivatives of $\Phi$ at $(0,\eta)$ read
\begin{align*}
D_r\Phi(0,\eta)[h]
&=
\big(
h_1 +\eta \Omega_1 *(u_c \nabla\cdot B(h_1,h_2)) - \alpha_{-1} \Omega_1*h_2,\ 
h_2 +\eta \Omega_{-1} *(v_c \nabla\cdot B(h_1,h_2)) - \alpha_1 \Omega_{-1}*h_1
\big),\\
D_{r\eta}\Phi(0,\eta)[h]
&=
\big(
\Omega_1*(u_c \nabla\cdot  B(h_1,h_2)),\ 
\Omega_{-1}*(v_c \nabla\cdot B(h_1,h_2))
\big)
\end{align*}
for 
$h=(h_1, h_2) \in X_s$.
Since $\Omega_\ell:L^2\to H^2$ and $\nabla\Omega_\ell:L^2\to H^1$ are smoothing,
they are compact as maps $L^2\to L^2$ by Sobolev embedding.
Thus $D_r\Phi(0,\eta)=I-T$ with $T$ compact, and $D_r\Phi(0,\eta)$ is Fredholm of index $0$.

To analyze its kernel on $X_s$, expand in the cosine basis:
\[
h_1=\sum_{k\in\N_0^d\setminus\{0\}} h_{1,k} w_k,
\quad
h_2=\sum_{k\in\N_0^d\setminus\{0\}} h_{2,k} w_k.
\]
Projecting $D_r\Phi(0,\eta)[h]$ onto each $w_k$ yields
\[
D_r\Phi(0,\eta)[h]
=\sum_{k\in\N_0^d\setminus\{0\}}
M(k,\eta)\begin{pmatrix}h_{1,k}w_k\\ h_{2,k}w_k\end{pmatrix},
\]
where
\[
M(k,\eta)=
\begin{pmatrix}
\displaystyle 1 -  \frac{ \eta u_c |2\pi k|^2 \widetilde{U}(k)}{D_1(k)}
&
\displaystyle  - \frac{\alpha_{-1} +  \eta u_c |2\pi k|^2 \widetilde{V}(k)}{D_1(k)}
\\[4mm]
\displaystyle  -\frac{\alpha_1 +  \eta v_c |2\pi k|^2 \widetilde{U}(k)}{D_{-1}(k)}
&
\displaystyle 1 - \frac{ \eta v_c |2\pi k|^2 \widetilde{V}(k)}{D_{-1}(k)}
\end{pmatrix}.
\]
A nontrivial kernel in mode $k$ occurs when $\det M(k,\eta)=0$.
Expanding $\det M(k,\eta)=0$ gives
\[
0
=\frac{D_1 D_{-1} - \alpha_1 \alpha_{-1} - \eta|2\pi k|^2\big[u_c \widetilde{U} (D_{-1}+ \alpha_1) + v_c \widetilde{V} (D_1 + \alpha_{-1})\big]}{D_1 D_{-1}},
\]
using $\alpha_1u_c=\alpha_{-1}v_c$, where we omit the explicit dependence on $k$ when clear.
Thus the critical value for mode $k$ is
\begin{align*}
\eta_k
&=\frac{D_1(k)D_{-1}(k)-\alpha_1\alpha_{-1}}
{|2\pi k|^2\big[v_c \widetilde V(k)(D_1(k)+\alpha_{-1})
+u_c \widetilde U(k)(D_{-1}(k)+\alpha_1)\big]}\\
&=\frac{(\alpha_1+\alpha_{-1})(D_1D_{-1}-\alpha_1\alpha_{-1})}
{|2\pi k|^2\big[\alpha_1\widetilde V(D_1+\alpha_{-1})+\alpha_{-1}\widetilde U(D_{-1}+\alpha_1)\big]}.
\end{align*}
Moreover,
\[
D_1 D_{-1} - \alpha_1\alpha_{-1}
= \frac{|2\pi k|^2\alpha_{-1}}{\beta_1} + \frac{|2\pi k|^2\alpha_1}{\beta_{-1}} + \frac{|2\pi k|^4}{\beta_1\beta_{-1}}>0,
\]
so $\eta_k>0$ is equivalent to
\[
\alpha_1\,\widetilde V(k)\,(D_1(k)+\alpha_{-1})
+\alpha_{-1}\,\widetilde U(k)\,(D_{-1}(k)+\alpha_1)>0.
\]
Under assumptions \textup{(i)}--\textup{(ii)}, choosing $\eta_*:=\eta_{k^*}$ yields
\[
\ker D_r\Phi(0,\eta_*)=\mathrm{span}\{(h_{1,k^*}w_{k^*},h_{2,k^*}w_{k^*})\},
\]
where $(h_{1,k^*},h_{2,k^*})$ is any nonzero null vector of $M(k^*,\eta_*)$.
This verifies Condition \textup{(1)}.

\medskip
\noindent{\it Verification of Condition (2).}
Recall
\[
D_{r\eta}\Phi(0,\eta)[h]
=\big( \Omega_1*(u_c\nabla \cdot B(h)),\ \Omega_{-1}*(v_c\nabla \cdot B(h)) \big).
\]
For $h_j=\sum_k h_{j,k}w_k$ we have
\[
(\nabla\cdot B(h))\,\tilde{}(k)
= -|2\pi k|^2\big(\widetilde U(k)h_{1,k}+\widetilde V(k)h_{2,k}\big),
\quad k\in\N_0^d\setminus\{0\}.
\]
Therefore,
\begin{align}\label{eq:Fourier-DrEtaPhi-clean}
\widetilde{D_{r\eta}\Phi(0,\eta)[h]}(k)
=
\left(
- u_c \frac{|2\pi k|^2\big(\widetilde{U}(k) h_{1,k}+\widetilde{V}(k) h_{2,k}\big)}{D_1(k)},\,
- v_c \frac{|2\pi k|^2\big(\widetilde{U}(k) h_{1,k}+\widetilde{V}(k) h_{2,k}\big)}{D_{-1}(k)}
\right).
\end{align}

Let $\tilde v_0=(h_{1,k^*}w_{k^*},h_{2,k^*}w_{k^*})\in\ker D_r\Phi(0,\eta_*)$ with $\|\tilde v_0\|_{L^2\times L^2}=1$.
We choose \[(h_{1,k^*},h_{2,k^*})=(-M_{12}(k^*,\eta_*),M_{11}(k^*,\eta_*)),\]  and concentrate on mode $k^*$
\[h_{1,k}=h_{1, k^*}\delta_{k,k^*}, \quad h_{2,l}=h_{2, k^*}\delta_{l,k^*}.\]
Plugging this into \eqref{eq:Fourier-DrEtaPhi-clean} yields
{
\[
\widetilde{D_{r\eta}\Phi(0,\eta_*)[\tilde v_0]}(k)
=
\left(
- u_c \frac{|2\pi k|^2\big(\widetilde{U}(k)h_{1, k}+\widetilde{V}(k)h_{2, k}\big)}{D_1(k)} \delta_{k,k^*},\, 
- v_c \frac{|2\pi k|^2\big(\widetilde{U}(k)h_{1, k}+\widetilde{V}(k)h_{2, k}\big)}{D_{-1}(k)} \delta_{k,k^*}
\right).
\]
}
Thus we have
\[
D_{r\eta}\Phi(0,\eta_*)[\tilde v_0]
=
\left(
- u_c \frac{|2\pi k^*|^2\big(\widetilde{U}(k^*)h_{1,k^*}+\widetilde{V}(k^*)h_{2,k^*}\big)}{D_1(k^*)} w_{k^*},\ 
- v_c \frac{|2\pi k^*|^2\big(\widetilde{U}(k^*)h_{1,k^*}+\widetilde{V}(k^*)h_{2,k^*}\big)}{D_{-1}(k^*)} w_{k^*}
\right).
\]
Hence, using $\|w_{k^*}\|_{L^2}=1$,
\begin{align}\label{eq:transv-ip-clean}
\Big\langle D_{r\eta}\Phi(0,\eta_*)[\tilde v_0], \tilde v_0\Big\rangle_{L^2\times L^2}
= -\frac{|2\pi k^*|^2}{D_1(k^*)D_{-1}(k^*)}\, \mathcal Q,
\end{align}
where
\[
\mathcal Q
=
u_cD_{-1}(k^*)\big(\widetilde U(k^*) h_{1,k^*}^2+\widetilde V(k^*)h_{1,k^*}h_{2,k^*}\big)
+v_cD_1(k^*)\big(\widetilde U(k^*) h_{1,k^*}h_{2,k^*}+\widetilde V(k^*)h_{2,k^*}^2\big).
\]
In particular, if the right-hand side of \eqref{eq:transv-ip-clean} is nonzero, then
 $D_{r\eta}\Phi(0,\eta_*)[\tilde v_0]\notin \ker(D_r\Phi(0,\eta_*))^\perp$. 
and the transversality condition holds once $\ker(D_r\Phi)^*=\ker D_r\Phi$.
This condition is equivalent to the matrix $M(k^*, \eta_*)$ being symmetric.
 Finally, the block symmetry  $M_{12}(k^*,\eta_*)=M_{21}(k^*,\eta_*)$ reads
\[\frac{\alpha_{-1}}{D_1(k^*)}+\frac{\eta_* u_c |2\pi k^*|^2 \widetilde{V}(k^*)}{D_1(k^*)}
=
\frac{\alpha_1}{D_{-1}(k^*)}+\frac{\eta_* v_c |2\pi k^*|^2 \widetilde{U}(k^*)}{D_{-1}(k^*)},
\]
that is,
\begin{align}\label{symmetric}
\alpha_{-1} D_{-1}(k^*) - \alpha_1 D_1(k^*) = \eta_* |k^*|^2 (v_c \widetilde{U}(k^*)D_1(k^*) - u_c \widetilde{V}(k^*)D_{-1}(k^*)).
\end{align}

With the criticality relation for $\eta_*$,
\eqref{symmetric} is equivalent to 
\begin{align*}
&\lt(\alpha_{-1} D_{-1} - \alpha_1 D_1\rt)
\lt(\alpha_1 \widetilde{V} (D_1+\alpha_{-1}) + \alpha_{-1} \widetilde{U} (D_{-1} + \alpha_1)\rt)
\cr
&\quad = (\alpha_1 + \alpha_{-1})(D_1 D_{-1} - \alpha_1 \alpha_{-1})\lt(v_c \widetilde{U} D_1 - u_c \widetilde{V} D_{-1} \rt)\cr
&\quad =  (D_1 D_{-1} - \alpha_1 \alpha_{-1})\lt(\alpha_1 \widetilde{U} D_1 - \alpha_{-1} \widetilde{V} D_{-1}\rt)
\end{align*}
evaluated at $k=k^*$, which is the condition \textup{(iii)}.

%
%
%
%
%
%

\section*{Acknowledgments}
The authors were supported by the National Research Foundation of Korea(NRF) grant funded by the Korea government (MSIT): M. Chae  by RS-2023-00279920; Y.-P. Choi by No. 2022R1A2C1002820 and RS-2024-00406821.

%
%
%
%

\bibliographystyle{abbrv}
\bibliography{CC_2can}

\end{document}